\documentclass[11pt]{article}
\usepackage{amsmath, amssymb, algorithm, algpseudocode}
\usepackage{tikz}
\usetikzlibrary{calc}
 \usepackage{bbm}
 \usepackage[sort]{cite}
\usepackage{hyperref}
\usepackage{graphicx}
\usepackage{diagbox}
\usepackage{epsfig}
\usepackage{epstopdf}
\usepackage{amssymb,amsmath}
\usepackage{amsmath}
\usepackage{amsthm}
\usepackage{amsfonts}
\usepackage{psfrag}
\usepackage{rotating}
\usepackage{latexsym}
\usepackage{dsfont}
\usepackage{amsthm}
\usepackage{amssymb}
\usepackage{amsmath}
\usepackage[force,almostfull]{textcomp}
\usepackage{fancyvrb}
\usepackage{textcomp}
\usepackage{url}
\usepackage{comment}
\usepackage{ifdraft}
\usepackage{multirow}
\usepackage{rotating}
\usepackage{xspace}
\usepackage{array}
\usepackage{flushend}
\usepackage{amsthm}
\theoremstyle{plain}
\newtheorem{remark}{Remark}
\newtheorem{corollary}{Corollary}
\newtheorem*{remarks*}{Remarks}
\newtheorem{theorem}{Theorem }
\newtheorem{proposition}{Proposition}
\newtheorem{Corollary}{Corollary}
\newtheorem{assum}{Assumption}
\usepackage{enumerate}
\usepackage{authblk}
\usepackage[utf8]{inputenc}
\newlength\figureheight
\newlength\figurewidth
\usepackage{graphics}
\usepackage{pgfplots}
\usepackage{siunitx}
\usepackage{subcaption}
\pgfplotsset{compat=newest}
\pgfplotsset{plot coordinates/math parser=false}
\usepackage{scalefnt}

\newtheoremstyle{specialcasestyle}{1mm}{1mm}{\upshape}{}{\bfseries\upshape}{.}{0mm}{}
\theoremstyle{specialcasestyle}

\begin{document}

\title{Hierarchical Importance Sampling for Estimating Occupation Time for SDE Solutions}

\author[1]{Eya Ben Amar}
\author[2]{Nadhir Ben Rached}
\author[1,3]{Ra\'ul Tempone}

\affil[1]{Computer, Electrical and Mathematical Sciences \& Engineering Division (CEMSE), King Abdullah University of Science and Technology (KAUST), Thuwal, Saudi Arabia.}
\affil[2]{School of Mathematics, University of Leeds, Leeds, UK.}
\affil[3]{Alexander von Humboldt Professor in Mathematics for Uncertainty Quantification, RWTH Aachen University, Aachen, Germany.}

\date{}
\maketitle
\thispagestyle{empty}
\begin{abstract}
This study considers the estimation of the complementary cumulative distribution function of the occupation time (i.e., the time spent below a threshold) for a process governed by a stochastic differential equation. The focus is on the right tail, where the underlying event becomes rare, and using variance reduction techniques is essential to obtain computationally efficient estimates. Building on recent developments that relate importance sampling (IS) to stochastic optimal control, this work develops an optimal single-level IS (SLIS) estimator based on the solution of an auxiliary Hamilton–Jacobi–Bellman (HJB) partial differential equation (PDE). The cost of solving the HJB-PDE is incorporated into the total computational work, and an optimized trade-off between preprocessing and sampling is proposed to minimize the overall cost. The SLIS approach is extended to the multilevel setting to enhance efficiency, yielding a multilevel IS (MLIS) estimator. A key feature that emerges in the MLIS framework is that the single-level variance decreases to zero as the discretization level increases, a behavior that stems from the zero-variance property of the optimal control. This observation leads to the formulation of a necessary and sufficient condition under which MLIS outperforms SLIS. To ensure that this condition is satisfied, we introduce a smoothing of both the drift and the observable to enhance the variance decay rate, together with a novel common-likelihood MLIS formulation that preserves this decay under IS. 
The classical multilevel Monte Carlo complexity theory can be extended to accommodate settings in which the variance on coarse levels depends on the target accuracy, a phenomenon for which the variance-decay behavior in the IS setting is a key example. Notably, the total work complexity of MLIS can be better than than quadratic. Numerical experiments in the context of fade duration estimation demonstrate the benefits of the proposed approach and validate the theoretical results, including the analysis of the smoothing error, which is shown to be negligible. Consequently, the smoothed MLIS estimator remains unbiased for the occupation time while achieving significantly better performance than SLIS.

\end{abstract}

\section{Introduction}

This study considers a system in which a process evolves according to a stochastic differential equation (SDE), with the goal of characterizing the distribution of the occupation time, specifically, the time a process spends below a critical threshold. Occupation time has been analytically studied in the context of Markov chains using matrix-based techniques~\cite{sericola2000occupation}. In contrast, the present work addresses a more general setting in which the underlying dynamics are governed by SDEs with continuous state spaces, requiring advanced stochastic-simulation techniques for accurate estimation.

Occupation time arises naturally in numerous applications. For example, in wireless communications, occupation time corresponds to the fade duration (i.e., the total time a signal remains below a reliability threshold)~\cite{amar2025stochastic}. In financial mathematics, similar occupation-time functionals appear in the pricing of exotic derivatives, such as Parisian and corridor options, where payoffs depend on the cumulative time an asset price remains above or below a barrier~\cite{chesney1997brownian,choi2021occupation}. In biophysics, occupation times, also known as residence times, describe how long particles stay in potential wells~\cite{barkai2006residence}. 

A quantity of particular interest is the complementary cumulative distribution function (CCDF) of the occupation time, which measures the probability that the occupation time exceeds a prescribed duration. This study simulates sample paths of the underlying process by numerically solving the SDE and estimates this probability using Monte Carlo (MC) methods to approximate this distribution. In the right tail of the CCDF, where the prescribed duration approaches the final time, the event corresponds to the process remaining below the threshold for an exceptionally extended period. Such events are extremely rare, making their estimation a type of rare-event probability problem. In this regime, the crude MC method becomes ineffective, and variance reduction techniques are necessary to obtain accurate estimates.

Among such techniques, importance sampling (IS)~\cite{kroese2013handbook,biondini2015introduction} is widely used for rare-event estimation in SDEs. 
Several studies have established a connection between IS and stochastic optimal control (SOC)~\cite{hartmann2018importance,ben2023state,zhang2014applications,ben2024double,amar2025stochastic,hammouda2024automated,rached2026importance}, where the optimal change of measure is characterized by the solution of an auxiliary Hamilton–Jacobi–Bellman (HJB) partial differential equation (PDE). For example,~\cite{ben2024double} develops an optimal change of measure via SOC
to estimate rare-event probabilities for McKean--Vlasov SDEs, 
while~\cite{amar2025stochastic} applies a similar approach to estimate the CCDF of the fade duration 
in wireless communications. 
Although these studies have demonstrated substantial variance reduction and have analyzed the sampling cost of IS estimators, they do not account for the computational cost of obtaining the optimal change of measure because they solve the required auxiliary HJB-PDE off-line.

In the present study, the trade-off between the accuracy of the control \eqref{control} and the resulting variance reduction is investigated, and the cost of solving the auxiliary HJB-PDE~\eqref{PDE} is incorporated into a comprehensive analysis of the total computational work of the IS estimator. The analysis reveals that the variance of the IS estimator decreases with the target accuracy, resulting in a lower overall computational cost for IS compared to the crude MC method. 

In order to further improve the efficiency of the single-level IS (SLIS) estimator for rare-event probability estimation, 
we extend our previous work~\cite{amar2025stochastic} to the multilevel setting by combining IS with the 
multilevel MC (MLMC) method introduced by~\cite{giles2015multilevel}.
Several studies have investigated combining IS with the MLMC method. For example, one study \cite{ben2020importance} applied IS to address the high‑kurtosis phenomenon in the MLMC method for stochastic reaction networks, and another study \cite{ben2023adaptive,kebaier2018coupling} proposed variance‑reduction schemes based on suboptimal, levelwise constant controls. The most recent development \cite{ben2025multilevel} combines SOC-based IS with the MLMC method in a single‑likelihood multilevel IS (MLIS) framework, achieving substantial computational savings. In the context of Bayesian inverse problems, various methods, such as those in \cite{dodwell2015hierarchical, lykkegaard2023multilevel}, exploit coarse-level likelihood evaluations to inform fine-level sampling, demonstrating that reusing the likelihood structure across model hierarchies can be highly effective. Building on these insights and following the method in \cite{ben2025multilevel}, this work introduces a novel common-likelihood MLIS approach. This formulation preserves the variance-convergence rate of MLMC and reduces the total computational work required by MLIS.

A crucial feature that emerges when applying the proposed MLIS approach is that the variance at the coarse level decreases to zero as the desired accuracy is refined. This behavior contrasts with the standard MLMC setting, where the single‑level variance is assumed to be bounded across all discretization levels of the SDE. Similar variance‑decay phenomena have been observed in other settings, such as particle approximations of McKean--Vlasov SDEs, where the estimator variance decreases with the system size due to the propagation of chaos~\cite{bossy1997stochastic}. This observation motivated the work in this study to extend the classical complexity theorem of MLMC~\cite{giles2015multilevel} to a more general setting. This work analyzes several scenarios and demonstrates that, in certain cases, MLIS can achieve better computational rates than the standard MLMC method due to the decay of the single-level variance. 

A further contribution of this work is the formulation of a necessary and sufficient condition under which MLIS outperforms SLIS. To ensure that this condition is satisfied in the occupation time case, we introduce a smoothing of both the drift \eqref{drift} and the observable \eqref{observable}, which enhances the decay rate of the multilevel variance. Moreover, we show that the proposed common-likelihood MLIS formulation preserves this improved decay under IS. Since smoothing modifies the original problem, we also analyze the induced smoothing error and prove that it is negligible, ensuring that the resulting smoothed MLIS estimator remains unbiased for the occupation time.

The contributions of this paper are summarized as follows.
\begin{itemize}
    \item An optimal SLIS estimator is developed to estimate the right tail of the CCDF of the occupation time, exploiting the connection between IS and SOC.
    \item The preprocessing cost is incorporated into the total computational work of the IS estimator, allowing a quantitative analysis of the trade-off between the control accuracy and variance reduction for a minimal computational cost.
    \item The SLIS framework is extended to the MLIS approach by combining the proposed SLIS with the MLMC method to improve computational efficiency.
    \item A novel common-likelihood MLIS approach is introduced, including an analysis of its influence on the variance and its benefit in reducing the total work.
    \item A necessary and sufficient condition for MLIS to outperform SLIS is derived.
    \item An optimization strategy is proposed for the total work of MLIS that accounts for the sampling and preprocessing costs of solving the auxiliary HJB.
    \item The classical MLMC complexity theorem is extended to a more general setting, where the single-level variance decays as the accuracy requirement is refined, demonstrating that MLIS can achieve complexity rates of better than 2 in such cases.
    \item In the setting of occupation-time estimation, a smoothing procedure for both the drift and the observable is introduced to improve the variance decay and ensure that the necessary condition for MLIS to outperform SLIS is satisfied.
    \item The resulting smoothing error is analyzed and shown to be negligible, guaranteeing that the proposed smoothed MLIS estimator remains unbiased while achieving a lower computational cost than SLIS. 
\end{itemize}

The remainder of this paper is structured as follows. Section~\ref{Section 1} describes the problem setting. Next, Section~\ref{Section 2} details the SLIS method, presenting the construction of the optimal estimator, exploring the influence of the accuracy of the auxiliary HJB-PDE on the variance, and reporting the numerical results. Then, Section~\ref{Section 3} reviews the MLMC method and provides the sufficient and necessary conditions for its efficiency. Section~\ref{Section 4} introduces the MLIS approach, including the single-level likelihood formulation, the smoothing strategy, and the common-likelihood construction. Moreover, Section~\ref{Section 5} presents the analysis of the effect of the HJB-PDE accuracy on the efficiency of the MLIS method and the proposed optimal work formulation, providing a numerical example and investigating the work rate when the single-level variance decays. Finally, Section~\ref{error analysis} analyzes the impact of smoothing on the estimator, derives a quadratic characterization of the smoothing error, and establishes parameter-selection strategies ensuring that smoothing introduces only a negligible bias in the occupation-time setting.
\section{Problem Setting}
\label{Section 1}
First, let \(\boldsymbol{X}(s) \in \mathbb{R}^d\) be the solution to the following SDE:
\begin{equation}
\label{SDEgd}
\begin{cases}
d \boldsymbol{X}(s) = a(s, \boldsymbol{X}(s)) \, ds + b(s, \boldsymbol{X}(s)) \, dW(s), & 0 < s < T, \\
\boldsymbol{X}(0) = x_0,
\end{cases}
\end{equation}
where \(a: [0, T] \times \mathbb{R}^d \to \mathbb{R}^d\) denotes the drift coefficient, \(b: [0, T] \times \mathbb{R}^d \to \mathbb{R}^{d \times d}\) represents the diffusion coefficient, and \(W\) indicates the \(d\)-dimensional standard Brownian motion. Then, let \(h: \mathbb{R}^d \to \mathbb{R}\) be a given function. The occupation time \(Z(T)\) over the time interval \([0, T]\) is defined as the total amount of time during which \(h(\boldsymbol{X}(s))\) remains below a given threshold \(\gamma_{\text{th}}\). The quantity \(Z(T)\) is obtained as the solution of the following ordinary differential equation:
\begin{equation}
\label{ZODE}
\begin{cases}
dZ(s) = f(h(\boldsymbol{X}(s))) \, ds, & 0 < s < T, \\
Z(0) = 0,
\end{cases}
\end{equation}
where 
\begin{equation}
\label{drift}
f(x) := \mathbbm{1}_{\{x < \gamma_{\text{th}}\}}.
\end{equation}
The aim is to characterize the CCDF of the occupation time \(Z(T)\), that is,
\begin{equation}
\label{equ}
q_w=\mathbb{P}(Z(T) > w \mid \boldsymbol{X}(0) = {x}_0, Z(0) = 0) 
= \mathbb{E}\left[ g_w(Z(T)) \mid \boldsymbol{X}(0) ={x}_0, Z(0) = 0 \right],
\end{equation}
where 
\begin{equation}
\label{observable}
g_w(x) = \mathbbm{1}_{\{x > w\}} \qquad \text{for a given threshold \(w > 0\).}
\end{equation}
This formulation can be compactly described by the following coupled system:
\begin{equation}
\label{SDEgdplus}
\begin{cases}
d \boldsymbol{X}(s) = a(s, \boldsymbol{X}(s)) \, ds + b(s, \boldsymbol{X}(s)) \, dW(s), & 0 < s < T, \\
d Z(s) =  f(h(\boldsymbol{X}(s))) \, ds, & 0 < s < T, \\
\boldsymbol{X}(0) = x_0, \quad Z(0) = 0.
\end{cases}
\end{equation}

This paper applies the MC estimator to estimate the CCDF in~\eqref{equ} and adopts the Euler--Maruyama scheme to solve the SDE system~\eqref{SDEgdplus}  numerically. Let \(N \in \mathbb{N}\) denote the number of time steps, with step size \(\Delta t = \frac{T}{N}\). The time grid is defined as \(t_n = n \Delta t\), for \(n = 0, 1, \dots, N-1\). Therefore,
\begin{align}
\boldsymbol{X}_{n+1} &= \boldsymbol{X}_n + a(t_n, \boldsymbol{X}_n) \Delta t + b(t_n, \boldsymbol{X}_n) \Delta W_n, \label{EMX} \\
Z_{n+1} &= Z_n + f(h(\boldsymbol{X}_n))\Delta t, \label{EMZ}
\end{align}
where \(\Delta W_n \sim \mathcal{N}(0, \Delta t \, I_d)\) represents the independently and identically distributed Brownian motion increment. Starting from \(\boldsymbol{X}_0 = x_0\) and \(Z_0 = 0\), this scheme produces a discretized trajectory \((\boldsymbol{X}_n, Z_n)\) up to time \(T\). Then, the CCDF is approximated using \(M\) independent realizations of the discretized process as follows:
\begin{equation}
\label{MCestimator}
\mathcal{A}_{\mathrm{MC}} = \frac{1}{M} \sum_{m=1}^M g_w^{(m)},
\end{equation}
where $g_w^{(m)}=g_w\left(Z^{(m)}_N\right)$, and \(Z^{(m)}_N\) denotes the value of the occupation time at the final time step for the \(m\)th trajectory.

However, a significant challenge arises when estimating \(\mathbb{P}(Z(T) > w)\) for large values of \(w\) approaching \(T\). In this regime, the event \(\{Z(T) > w\}\) becomes rare, and standard MC sampling is inefficient. To address this, we employ IS, a variance reduction technique in which paths are simulated under an alternative probability measure that makes the rare event more likely \cite{kroese2013handbook}. 

\section{Single-Level Importance Sampling}
\label{Section 2}
\subsection{Optimal Importance Sampling estimator}
In our previous work~\cite{amar2025stochastic}, the authors developed an optimal IS estimator for~\eqref{equ} in the context of wireless communication systems by changing the measure in the path space of the Brownian motion driving the SDE~\eqref{SDEgdplus}. The change of measure was parametrized by a control function \(\zeta : [0,T] \times \mathbb{R}^d \times [0,T] \to \mathbb{R}^d\). The resulting controlled process satisfies the discretized dynamics:
\begin{align}
\tilde{\boldsymbol{X}}_{n+1} &= \tilde{\boldsymbol{X}}_n 
+ \left( a(t_n, \tilde{\boldsymbol{X}}_n) 
+ b(t_n, \tilde{\boldsymbol{X}}_n) \, \zeta(t_n, \tilde{\boldsymbol{X}}_n,\tilde{Z}_n ) \right) \Delta t 
+ b(t_n, \tilde{\boldsymbol{X}}_n) \,  \Delta W_n, \label{controlledEMX} \\
\tilde{Z}_{n+1} &= \tilde{Z}_n 
+ f(h(\tilde{\boldsymbol{X}}_n))\Delta t,
\qquad n=0,\cdots,N-1. \label{controlledEMZ}
\end{align}
The optimal control \(\zeta^*\) was obtained by minimizing the second moment of the estimator, resulting in a SOC problem. The optimal control is given in closed form as follows: 
\begin{equation}
\label{control}
\zeta^*(t,\boldsymbol{x},z) =  b(t,\boldsymbol{x}) \nabla \log v(t,\boldsymbol{x},z),
\end{equation}
where \(v : [0,T] \times \mathbb{R}^d \times [0,T] \to \mathbb{R}\) solves the Kolmogorov backward equation associated with~\eqref{SDEgdplus} (see \cite{amar2025stochastic}). We refer to this linear PDE as the auxiliary HJB-PDE.

This auxiliary HJB-PDE involves \( d \) spatial variables in addition to time; hence, solving it becomes computationally expensive when the state dimension \(d\) is large. A possible strategy to mitigate this cost is to derive a reduced SDE for the scalar process \(h(\boldsymbol{X}(s))\) and solve the corresponding Kolmogorov backward equation for the lower-dimensional system \((h(\boldsymbol{X}(s)), Z(s))\). 
Markovian projection is a technique that Gyöngy~\cite{gyongy1986mimicking} introduced to reduce the dimensionality of stochastic systems. This method constructs a lower-dimensional SDE, preserving the marginal distribution of a function of the original process. This approach has been used in recent work~\cite{amar2025stochastic,hammouda2024automated,rached2026importance}. In this context, this approach enables defining a one-dimensional process \(\bar{X}(s)\) that mimics the distribution of \(h(\boldsymbol{X}(s))\). The dynamics of \(\bar{X}(s)\) follow a Markovian SDE with drift \(\bar{a}: [0,T] \times \mathbb{R} \to \mathbb{R}\) and diffusion \(\bar{b} : [0,T] \times \mathbb{R} \to \mathbb{R}\), both derived using the Markovian projection lemma~\cite{gyongy1986mimicking}. This reduced formulation yields a surrogate value function \(\bar{v} : [0,T] \times \mathbb{R} \times [0,T] \to \mathbb{R}\) that approximates the original value function \(v\) as follows:

\[
\label{vapprpx}
v(t, \boldsymbol{x}, z) \approx \bar{v}(t, h(\boldsymbol{x}), z).
\]
The function \(\bar{v}(t,x,z)\) solves the following auxiliary HJB:
\begin{equation}
\label{PDE}
\left\{
\begin{aligned}
&\partial_t \bar{v} + \bar{a}(t,x) \, \partial_x \bar{v} + f(x) \, \partial_z \bar{v} + \frac{1}{2} \bar{b}^2(t,x) \, \partial_{x x} \bar{v} = 0, \quad 0 \leq t < T, \; x \in \mathbb{R}, \; z \in [0,T], \\
&\bar{v}(T, x, z) = g_w(z), \quad \text{for all } x \in \mathbb{R}, \; z \in [0,T].
\end{aligned}
\right.
\end{equation}
The IS must be applied to the high-dimensional SDE~\eqref{SDEgdplus}, not to its projected version, because the reduced drift and diffusion coefficients \(\bar{a}\) and \(\bar{b}\) cannot generally be computed exactly, and any approximation introduces bias. Therefore, the projected SDE is employed only to construct the control via the following approximation:
\begin{equation}
\label{OC}
\zeta^*(t,\boldsymbol{x},z) \approx b(t,\boldsymbol{x}) \, \partial_x \log \bar{v}(t, h(\boldsymbol{x}), z) \cdot \nabla_{\boldsymbol{x}} h(\boldsymbol{x}).
\end{equation}
Finally, the IS estimator is given by
\begin{equation}
\label{ISestimator}
\mathcal{A}_{\mathrm{IS}} = \frac{1}{M_{\mathrm{IS}}} \sum_{m=1}^{M_{\mathrm{IS}}} 
\tilde{g}_w^{(m)},
\end{equation}
where
\begin{equation}
\label{tildeg}
\tilde{g}_w^{(m)} = g_w\left(\tilde{Z}_N^{(m)}\right) L_N^{(m)},
\end{equation}
and the likelihood ratio \( L_N^{(m)} \) is given by
\begin{equation}
\label{likelihood}
L_N^{(m)} = \prod_{n=0}^{N-1} \exp \left\{
  -\frac{1}{2} \Delta t \left\| \zeta^*(t_n, \tilde{\boldsymbol{X}}_n^{(m)}, \tilde{Z}_n^{(m)}) \right\|^2 
  - \left\langle \Delta W_n^{(m)}, \zeta^*(t_n, \tilde{\boldsymbol{X}}_n^{(m)}, \tilde{Z}_n^{(m)}) \right\rangle
\right\},
\end{equation}
where \( (\tilde{\boldsymbol{X}}_n^{(m)}, \tilde{Z}_n^{(m)}) \) denotes the \( m \)th sample path of the controlled process at time step \( n \).

The estimation of \(q_w = \mathbb{P}(Z(T) > w)\) using the unbiased IS estimator \(\mathcal{A}_{\mathrm{IS}}\) is affected by two types of error: a discretization error arising from the time discretization parameter \(N\), and a sampling error due to the finite number of trajectories \(M_{\mathrm{IS}}\). This approach decomposes the total relative error of the IS estimator \(\mathcal{A}_{\mathrm{IS}}\) into bias and statistical components:
\begin{equation}
\frac{\left| q_w - \mathcal{A}_{\mathrm{IS}} \right|}{q_w} 
\leq 
\underbrace{ \frac{ \left| q_w - \mathbb{E}[\mathcal{A}_{\mathrm{IS}}] \right| }{q_w} }_{= \epsilon_b}
+ 
\underbrace{ \frac{ \left| \mathbb{E}[\mathcal{A}_{\mathrm{IS}}] - \mathcal{A}_{\mathrm{IS}} \right| }{q_w} }_{= \epsilon_s},
\end{equation}
where \(\epsilon_b\) and \(\epsilon_s\) denote the relative bias and statistical error, respectively.
The bias \(\epsilon_b\) arises from the Euler--Maruyama time discretization and satisfies the following~\cite{kloeden1992stochastic}:
\begin{equation}
\epsilon_b = \frac{C_b}{q_w N},
\end{equation}
where \(C_b > 0\) can be estimated numerically. The statistical error \(\epsilon_s\) follows from the central limit theorem~\cite{asmussen2007stochastic}:
\begin{equation}
\epsilon_s = C \frac{\sqrt{V_{N}^{\mathrm{IS}}}}{q_w \sqrt{M_{\mathrm{IS}}}},
\end{equation}
where \(C\) represents the confidence constant, and \(V_{N}^{\mathrm{IS}}\) indicates the variance of the IS estimator:
\begin{equation}
\label{VIS}
V_{N}^{\mathrm{IS}} = \operatorname{Var} \left[ g_w(\tilde{Z}_N) 
\prod_{n=0}^{N-1} \exp \left(
  -\frac{1}{2} \Delta t \left\| \zeta^*(t_n, \tilde{\boldsymbol{X}}_n, \tilde{Z}_n) \right\|^2 
  - \langle \Delta W_n, \zeta^*(t_n, \tilde{\boldsymbol{X}}_n, \tilde{Z}_n) \rangle 
\right) \right].
\end{equation}

An optimal number of time steps \(N_{\mathrm{opt}}\) and optimal number of samples \(M_{\mathrm{IS}}\) are required to achieve a target relative tolerance \(\mathrm{TOL}\), given by
\begin{equation}
N_{\mathrm{opt}} = \frac{2 C_b}{q_w \, \mathrm{TOL}}, \label{eq:Nopt}
\end{equation}
\begin{equation}
M_{\mathrm{IS}} = \left( \frac{2 C}{q_w \, \mathrm{TOL}} \right)^2 V_{N_{\mathrm{opt}}}^{\mathrm{IS}}. \label{eq:Mopt}
\end{equation}
The computational cost associated with generating \(M_{\mathrm{IS}}\) samples of the controlled SDE is referred to as the sampling work:
\begin{equation}
\label{samplingwork}
\text{Work}_\text{sampling}^{\text{IS}} = C_{\text{SDE}} \cdot M_{\mathrm{IS}} \cdot N_{\mathrm{opt}},
\end{equation}
where \( C_{\text{SDE}} \) denotes the cost per sample per time step for solving the controlled SDE.
\subsection{Auxiliary HJB-PDE Accuracy Effect on IS Estimator Variance}
\label{section auxiliary HJB}
Previous studies~\cite{ben2022single,ben2023learning,awad2013zero,amar2025stochastic} have focused on constructing the IS estimator and have demonstrated the effectiveness of this approach for reducing variance. However, these studies have not examined in detail how accurately the optimal control is computed in practice, nor have they accounted for the cost of computing the control in their analyses. This section incorporates the auxiliary HJB-PDE cost into the efficiency analysis, enabling a fair and complete assessment of the overall performance of optimal IS. The proposed approach advances the method by quantifying the effect of the HJB-PDE solver accuracy on variance reduction and by introducing an optimized framework that balances preprocessing and sampling costs.

The optimal control \(\zeta^*\) is derived from the solution \(\bar{v}\) of the auxiliary HJB~\eqref{PDE}. However, \(\bar{v}\) is computed only on a discrete grid, and the evaluation of \(\zeta^*(t_n, \tilde{\boldsymbol{X}}_n^{(m)}, \tilde{Z}_n^{(m)})\) at off-grid points requires interpolation or extrapolation of the discrete solution at \((t_n, h(\tilde{\boldsymbol{X}}_n^{(m)}), \tilde{Z}_n^{(m)})\). Therefore, the accuracy of \(\zeta^*\) is affected by the auxiliary HJB-PDE discretization error, denoted by \(\epsilon_{\mathrm{PDE}}\). Therefore, the variance of the IS estimator~\eqref{VIS} is influenced by two factors: the time discretization parameter \(N\) and auxiliary HJB-PDE solver accuracy \(\epsilon_{\mathrm{PDE}}\). As \(N\) increases, the discrete control \(\zeta^*\) better approximates the continuous optimal control, which is known to yield zero variance~\cite{amar2025stochastic}. Similarly, improving \(\epsilon_{\mathrm{PDE}}\) enhances the quality of the control, provided \(N\) is large enough. In summary, the variance \(V_{N}^{\mathrm{IS}}\) depends on the SDE discretization level \(N\) and the auxiliary HJB-PDE discretization error \(\epsilon_{\mathrm{PDE}}\), and it decreases as either aspect is improved. Therefore, this work defines $V_{N}^{\mathrm{IS}} = V_{N}^{\mathrm{IS}}(\epsilon_{\mathrm{PDE}}).$

This work aims to achieve the lowest possible variance to reduce the computational effort required for sampling \eqref{samplingwork}. The cost of solving the auxiliary HJB-PDE should be controlled to prevent it from becoming prohibitively high. This approach introduces a trade-off between the variance reduction and computational cost. The objective is to determine the minimal achievable variance without incurring excessive auxiliary HJB-related costs. Let \(\text{Work}_\text{PDE}(\epsilon_{\mathrm{PDE}})\) denote the computational cost of solving the auxiliary HJB~\eqref{PDE}  with accuracy \(\epsilon_{\mathrm{PDE}}\). This paper proposes  Algorithm \ref{algSLIS} to determine the optimal total work:
\begin{equation}
\label{totalWork}
\text{Work}_{\mathrm{IS}}(V_{N}^{\mathrm{IS}},\epsilon_{\mathrm{PDE}}) = \text{Work}_\text{sampling}^{\mathrm{IS}}(V_{N}^{\mathrm{IS}}) + \text{Work}_\text{PDE}(\epsilon_{\mathrm{PDE}}).
\end{equation}

\begin{algorithm}
\caption{Optimization of total work of importance sampling}
\label{algSLIS}
\begin{algorithmic}[1]
\State \textbf{Input: } Target tolerance \( \mathrm{TOL} \), initial auxiliary HJB--PDE accuracy \( \epsilon_{\mathrm{PDE}}^0 \)
\State Compute \( N_{\mathrm{opt}} \) using~\eqref{eq:Nopt}
\State Set \( \text{Work}_{\mathrm{IS}}^{\mathrm{opt}} \gets \infty \) and \( \epsilon_{\mathrm{PDE}}\gets \epsilon_{\mathrm{PDE}}^0 \)
\While{true}
    \State Compute variance \( V = V_{N_{\mathrm{opt}}}^{\mathrm{IS}}(\epsilon_{\mathrm{PDE}}) \)
    \State Compute total work: \( \text{Work}_{\mathrm{IS}} = \text{Work}_{\mathrm{sampling}}^{\mathrm{IS}}(V) + \text{Work}_{\mathrm{PDE}}(\epsilon_{\mathrm{PDE}}) \)
    \If{ \( \text{Work}_{\mathrm{IS}} < \text{Work}_{\mathrm{IS}}^{\mathrm{opt}} \) }
        \State \( \text{Work}_{\mathrm{IS}}^{\mathrm{opt}} \gets \text{Work}_{\mathrm{IS}} \)
        \State \( \epsilon_{\mathrm{PDE}}^{\mathrm{opt}} \gets \epsilon_{\mathrm{PDE}} \)
        \State Refine auxiliary HJB--PDE by reducing \( \epsilon_{\mathrm{PDE}} \)
    \Else
        \State \textbf{break}
    \EndIf
\EndWhile
\State \textbf{Return: } \( \epsilon_{\mathrm{PDE}}^{\mathrm{opt}}, \text{Work}_{\mathrm{IS}}^{\mathrm{opt}} \)
\end{algorithmic}
\end{algorithm}
 \begin{remark}
In practice, HJB-PDE work can be predicted from its known scaling with mesh refinement, and the sampling work can be estimated by extrapolating the variance with respect to the HJB-PDE accuracy for a fixed~\(\mathrm{TOL}\). This method enables the avoidance of solving the HJB-PDE at every step of the optimization loop. Instead, the equation is solved only once, after the optimal accuracy \( \epsilon_{\mathrm{PDE}}^{\mathrm{opt}} \) is identified.
\end{remark}

\subsection{Numerical Results}
\label{section num}
This section addresses the estimation of the fade duration in wireless communication systems using an optimal IS estimator. The fade duration is the total time interval during which the signal power remains below a threshold. The instantaneous power is modeled as \(I^2 + Q^2 \), where \( I \) and \( Q \) represent in-phase and quadrature components of the signal, respectively. Following prior work~\cite{feng2007stochastic,charalambous1999stochastic}, these components were modeled as solutions to the coupled Ornstein--Uhlenbeck processes. This work focuses on the Rice fading scenario, where \( I \) and \( Q \) satisfy the SDE system:
\begin{equation}
\label{ricedynamics}
\left\{
\begin{aligned}
dI(s) &= k (\theta - I(s)) \, ds + \beta \, dW^{(I)}(s), \\
dQ(s) &= k (\theta - Q(s)) \, ds + \beta \, dW^{(Q)}(s), \quad 0 < s < T, \\
I(0) &= I_0, \\
Q(0) &= Q_0,
\end{aligned}
\right.
\end{equation}
where \( k, \theta, \beta \in \mathbb{R} \), and \( W^{(I)} \) and \( W^{(Q)}\) are independent Brownian processes. A previous study~\cite{amar2025stochastic} introduced a surrogate process \( \bar{X}(s) \) designed to approximate the distribution of the signal power \( X \), which evolves according to the SDE:
\begin{equation}
\label{MPSDE}
\left\{
\begin{aligned}
d\bar{X}(s) &= \bar{a}(s, \bar{X}(s)) \, ds + \bar{b}(s, \bar{X}(s)) \, dW(s), \quad 0 < s < T, \\
\bar{X}(0) &= I_0^2+Q_0^2,
\end{aligned}
\right.
\end{equation}
where the drift and diffusion terms \( \bar{a}\) and \(\bar{b} \) are derived via Markovian projection, as detailed in~\cite{amar2025stochastic}. This setting corresponds to the general framework outlined in the previous section, where \( \boldsymbol{X} = (I, Q) \) and \( h(\boldsymbol{x}) = x_1^2 + x_2^2 \). In this formulation, an IS estimator is constructed to estimate the CCDF of the fade duration as described in~\eqref{ISestimator}:
$Z(T) = \int_0^T f(h(\boldsymbol{X}(s))) \, ds.$
This work emphasizes quantifying the influence of the accuracy of the auxiliary HJB-PDE solution on the variance reduction achieved by the estimator.

The optimal control is obtained from~\eqref{OC} by solving the auxiliary HJB-PDE~\eqref{PDE}. This work discretizes the auxiliary HJB-PDE using an upwind scheme for the first derivatives in \(x\) and \(z\) to ensure stability in advection-dominated regions \cite{langtangen2017finite} and applies a central difference scheme for the second derivative in \(x\). These choices yield first-order accuracy in space in the \(x\) and \(z\) directions. Time integration is performed using an implicit--explicit splitting method \cite{leveque2007finite}, treating the \(z\)-advection term explicitly to preserve discontinuity and treating the \(x\)-advection-diffusion term implicitly for stability. This scheme is first-order accurate in time. The computational domains in \(t\), \(x\), and \(z\) are uniformly discretized using \(P\) steps in each direction, ensuring stability according to the eigenvalue analysis. Appropriate boundary conditions in \(z\) are derived from the analysis of the auxiliary HJB-PDE characteristics, and nonreflective extrapolation is applied in \(x\). The system is then solved backward in time from the terminal condition to compute the value function and associated optimal control. Therefore, the discretization error can be approximated as \( \varepsilon_{\text{PDE}} = \mathcal{O} (P^{-1} )\). Given the implicit--explicit scheme, the computational cost of solving the auxiliary HJB-PDE can be expressed as follows:
\begin{equation}
\label{PDEWork}
\text{Work}_{\text{PDE}} = C_{\text{PDE}} P^3,
\end{equation}
where \( C_{\text{PDE}} \) denotes a numerically determined constant. 

For all following experiments, the model parameters were set as follows: \( T = 5 \), \( I_0 = Q_0 = 1 \), \( k = 0.25 \), \( \theta = 0.2 \), \( \beta = 0.375 \), and threshold \( \gamma_{\text{th}} = 0.25 \). This approach solves the auxiliary HJB-PDE for several values of \( P \) to validate the expression in~\eqref{PDEWork}. Figure~\ref{costpde} plots the resulting computational cost as a function of \( P \). The constant \( C_{\text{PDE}} \approx 7 \times 10^{-7} \) is estimated from this analysis.
\begin{figure}[ht] 
\begin{center}
\includegraphics[scale = 0.4]{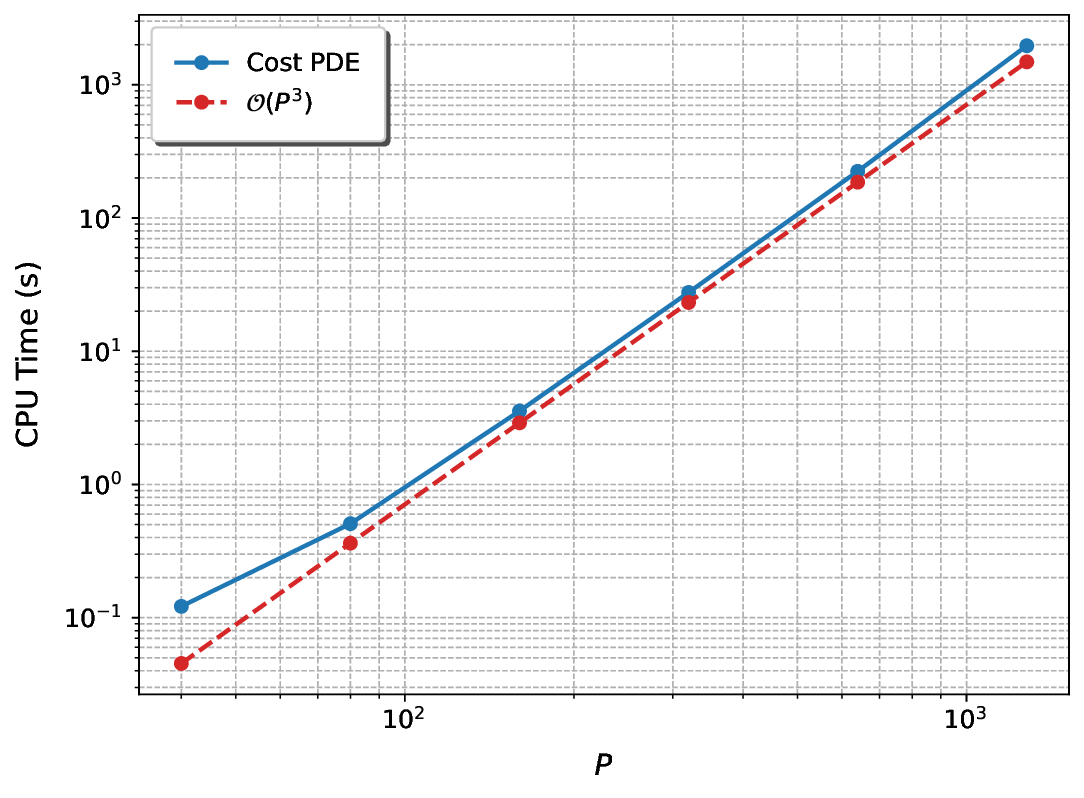} 
\caption{Auxiliary HJB-PDE solver cost as a function of the discretization parameter \( P \).}
\label{costpde}
\end{center}
\end{figure} 

Figure~\ref{varianceSLIS} displays the variance \( V_{N}^{\mathrm{IS}}\) for several values of \( P \) to examine how the auxiliary HJB-PDE discretization parameter \( P \) influences the quality of the computed control in terms of variance reduction and to assess the effect of \( N \). The results reveal that variance reduction improves with an increasing \( P \), if \( N \) is sufficiently large; likewise, it improves with an increasing \( N \), if \( P \) is sufficiently large. Variance saturation was observed for a large \( N \), due to the limiting effect of the auxiliary HJB-PDE discretization error on control accuracy. This saturation is postponed when the auxiliary HJB-PDE is solved more accurately (i.e., for larger values of \( P \)).
\begin{figure}[ht] 
\begin{center}
\includegraphics[scale = 0.4]{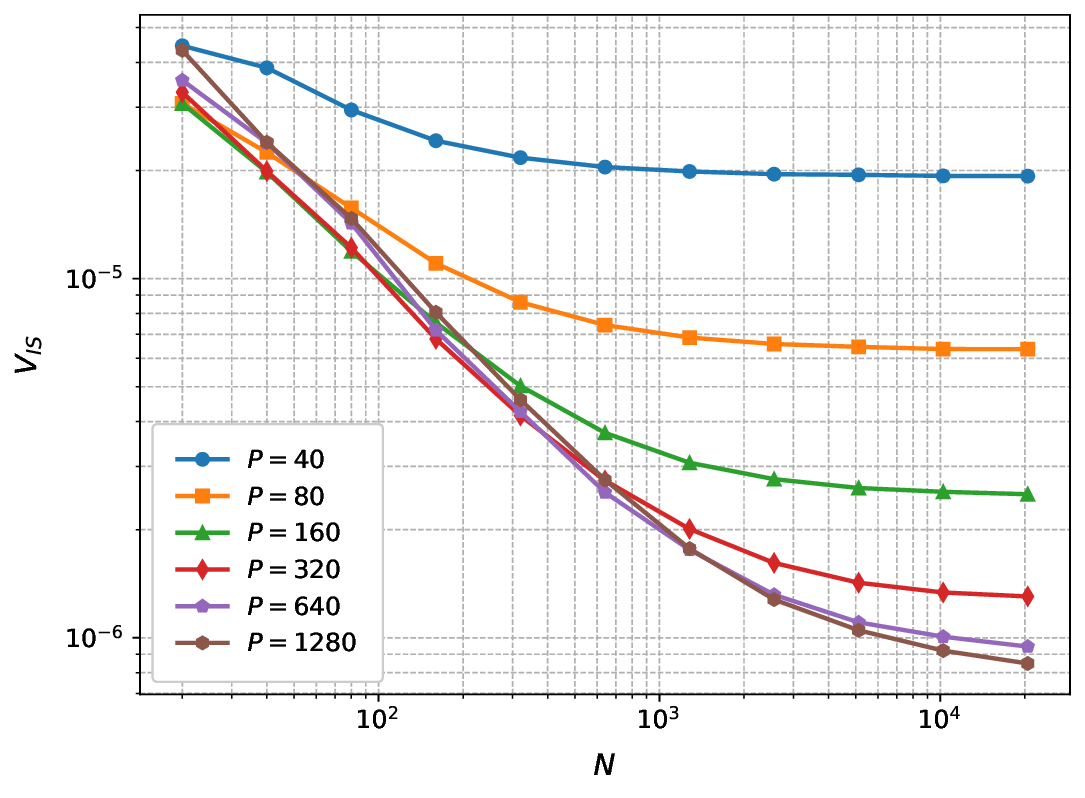} 
\caption{Importance sampling variance \( V_{N}^{\mathrm{IS}} \) for various auxiliary HJB-PDE discretizations \( P \) for $q_w=2 \times 10^{-3}$.}
\label{varianceSLIS}
\end{center}
\end{figure}

To estimate \( q_w \) with relative accuracy \( \text{TOL} \), \( N_{\mathrm{opt}} \) as defined in~\eqref{eq:Nopt} is required. The optimal number of discretization steps \( P_{\mathrm{opt}} \) for solving the auxiliary HJB--PDE depends on \( N_{\mathrm{opt}} \), and thus on \( \text{TOL} \). Algorithm~\ref{algSLIS} determines the optimal auxiliary HJB--PDE accuracy and corresponding \( P_{\mathrm{opt}} \) that minimizes the total computational work of the IS estimator.
Figure~\ref{SLIS} plots the optimal IS variance \( V_{N_{\mathrm{opt}}}^{\mathrm{IS}}(P_{\mathrm{opt}}) \) and the total computational work~\eqref{totalWork} as functions of \( \text{TOL} \), compared with the cost of estimating \( q_w \) using the crude MC method. This work estimates the constants \( C_{\mathrm{SDE}} \) and \( C_b \) in the sampling cost~\eqref{samplingwork} numerically as \( C_{\mathrm{SDE}} \approx 1.3 \times 10^{-7} \) and \( C_b = 0.02 \).

\begin{figure}[ht]
\begin{center}
\includegraphics[width=2.8in]{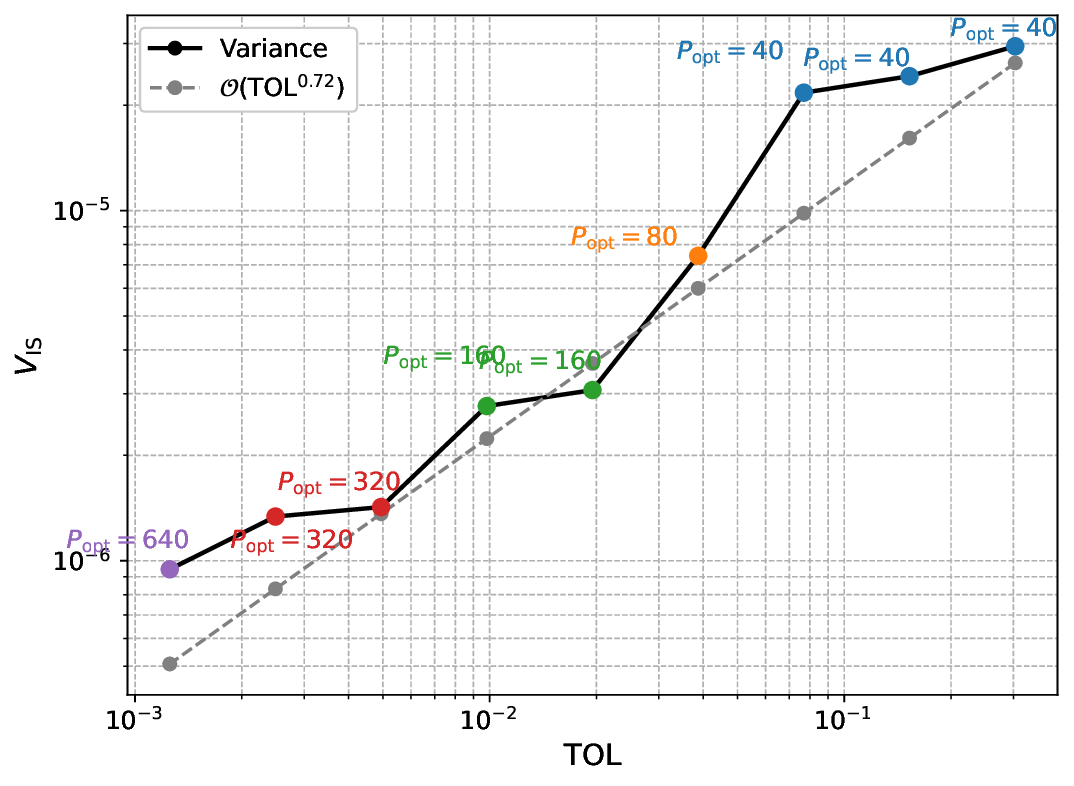}
\hspace{-0.1cm}
\includegraphics[width=2.8in]{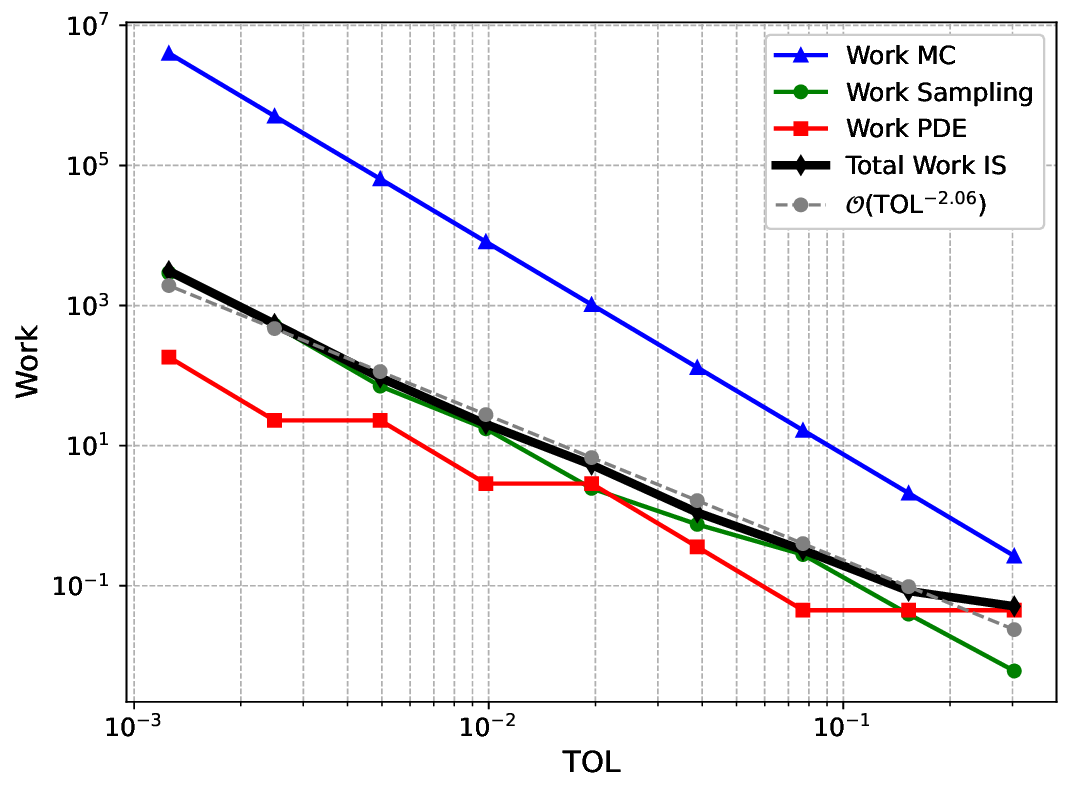}
\caption{Importance sampling variance and total work versus tolerance \( \text{TOL} \) for $q_w=2 \times 10^{-3}$.}
\label{SLIS}
\vspace{-3mm}
\end{center}
\end{figure}
The variance reduction illustrated in Figure~\ref{SLIS} is driven by the increase in \( N_{\mathrm{opt}} \) and \( P_{\mathrm{opt}} \) as \( \mathrm{TOL} \) decreases. This results in a variance decay rate of approximately 0.8, improving the total work complexity to 2.2 for the IS estimator, compared to a rate of 3 for the standard MC method.
These results highlight the importance of accounting for the auxiliary HJB-PDE cost to evaluate fully the efficiency of optimal IS for rare-event estimation. To the best of our knowledge, this is the first work to incorporate preprocessing time explicitly into the cost model and propose an optimized strategy to minimize the total work. In contrast, previous studies have typically solved the auxiliary HJB-PDE off-line. 
\section{Multilevel Monte Carlo}
\label{Section 3}
The SLIS method focuses on minimizing the cost of estimating a rare event probability while maintaining a fixed relative statistical error. However, achieving a small total error also requires controlling the bias, involving solving the SDE with \( N_{\mathrm{opt}} \) time steps, a task that can become computationally expensive, especially for small tolerances. This work employs the MLMC method to reduce the computational cost required to achieve a given tolerance, addressing this problem. By combining simulations across multiple discretization levels, the MLMC method shifts most of the computational effort to the coarse level, significantly reducing the overall cost. The following section considers the general MLMC framework and establishes the necessary and sufficient conditions under which the MLMC method outperforms the single-level MC (SLMC) method.
\subsection{Framework}
This work estimates the quantity \( q_w = \mathbb{E}[g_w(Z(T))] \). In the MLMC framework, a hierarchy of discretizations is defined with \( N_\ell = N_0 \times 2^\ell \) time steps for the levels \( \ell = 0, 1, \dots, L \). If \( g_w^{(\ell)} \) denotes the estimator computed using \( N_\ell \) time steps, the MLMC estimator is defined as follows \cite{giles2015multilevel}:
\begin{equation}
\mathcal{A}_{\mathrm{MLMC}}  = \frac{1}{M_0} \sum_{m=1}^{M_0} g_w^{(0,m)} + \sum_{\ell=1}^L \frac{1}{M_\ell} \sum_{m=1}^{M_\ell} \left( g_w^{(\ell,m)} - g_w^{(\ell-1,m)} \right),
\end{equation}
where \( M_\ell \) represents the number of samples at level \( \ell \), and each difference \( g_w^{(\ell)} - g_w^{(\ell-1)} \) is simulated using the same Brownian path at levels~\(\ell\) and~\(\ell-1\) to reduce variance.

For the relative total error of the MLMC estimator to satisfy a fixed tolerance \( \mathrm{TOL} \), the optimal number of levels \( L_{\mathrm{opt}} \) must be selected such that the relative bias is less than \( \frac{\mathrm{TOL}}{2}  \). This is achieved by selecting
\begin{equation}
\label{lopt}
L_{\mathrm{opt}} = \log_2\left(\frac{N_{\mathrm{opt}}}{N_0}\right),
\end{equation}
where \( N_{\mathrm{opt}} \) denotes the optimal number of time steps required to control the bias, as defined in~\eqref{eq:Nopt}.
Once \( L_{\mathrm{opt}} \) is fixed, the number of samples \( M_\ell \) at each level \( \ell \) is set to ensure that the variance of the MLMC estimator satisfies
\begin{equation}
\mathrm{Var}[\mathcal{A}_{\mathrm{MLMC}}] \leq \left( \frac{q_w \mathrm{TOL}}{2 C} \right)^2,
\end{equation}
guaranteeing that the relative statistical error is below \( \frac{\mathrm{TOL}}{2} \). The optimal number of samples per level is given by the following~\cite{giles2015multilevel}:
\begin{equation}
M_\ell =
\left( \frac{2 C}{q_w \, \mathrm{TOL}} \right)^2
\begin{cases}
\displaystyle
\sqrt{ \frac{V_{0}}{C_{0}} } \left( \sqrt{V_{0} C_{0}} + \sum_{j=1}^{L_{\mathrm{opt}}} \sqrt{V_{j,j-1} C_{j,j-1}} \right), & \ell = 0, \\[12pt]
\displaystyle
\sqrt{ \frac{V_{\ell,\ell-1}}{C_{\ell,\ell-1}} } \left( \sqrt{V_{0} C_{0}} + \sum_{j=1}^{L_{\mathrm{opt}}} \sqrt{V_{j,j-1} C_{j,j-1}} \right), & \ell \geq 1,
\end{cases}
\end{equation}
where \( C_{0} \) and \( V_{0} \) represent the cost and variance, respectively, of a single sample at level 0, computed using \( N_0 \) time steps, and \( C_{\ell,\ell-1} \) and \( V_{\ell,\ell-1} \) denote the cost and variance, respectively, of one sample of the difference \( g_w^{(\ell)} - g_w^{(\ell-1)} \) for \( \ell = 1, \dots, L_{\mathrm{opt}} \).
In standard MLMC settings, \(V_{0}\) is typically assumed to be of the same order across discretizations, bounded away from zero. In contrast, this setting allows \(V_{0}\) to depend on the discretization parameter \(N_0\), potentially decreasing significantly as \(N_0\) increases. This behavior is relevant when applying IS, where the variance of the IS estimator depends on the number of time steps \(N\), as presented in Figure~\ref{varianceSLIS}.

\subsection{Sufficient and Necessary Condition for MLMC Advantages}
The sampling work of the MLMC estimator is given by the following:
\begin{equation}
\label{samplingworkMLMC}
\begin{aligned}
\text{Work}_\text{sampling}^{\text{MLMC}}&=M_0 C_{0} +\sum_{\ell=1}^{L_{\mathrm{opt}}} M_\ell C_{\ell,\ell-1}\\ &=\left( \frac{2 C}{q_w \, \mathrm{TOL}} \right)^2
\left(\sqrt{V_{0} C_{0}}+ \sum_{\ell=1}^{L_{\mathrm{opt}}} \sqrt{V_{\ell,\ell-1} C_{\ell,\ell-1}}\right)^2.
\end{aligned}
\end{equation}
In contrast, the work of the SLMC estimator is given by the following:
\begin{equation}
\label{eq:work_slmc}
\text{Work}_\text{sampling}^{\text{SLMC}} = \left( \frac{2 C}{q_w \, \mathrm{TOL}} \right)^2 V_{L_{\mathrm{opt}}} C_{L_{\mathrm{opt}}}.
\end{equation}
In the standard case where the single-level variance is bounded, the MLMC method achieves a significantly lower asymptotic computational cost than the SLMC method as \( L_{\mathrm{opt}} \) increases. However, in practical scenarios, this asymptotic regime may not be attained. In such cases, the choice of the initial discretization level \( N_0 \) is crucial, especially when \( L_{\mathrm{opt}} \) is small.

A starting level \( N_0 \) is considered favorable if initiating the MLMC method from this level yields a lower overall sampling cost than starting from a finer level, equivalent to the following condition:
\begin{equation}
\label{ineqML}
\sqrt{V_{1} \, C_{1}} > \sqrt{V_{0} \, C_{0}} + \sqrt{V_{1,0} \, C_{1,0}}.
\end{equation}
Generally, a starting level \( N_{\ell_0} = 2^{\ell_0} N_0 \) is considered advantageous if it satisfies the following:
\begin{equation}
\label{condML0}
\sqrt{V_{{\ell_0+1}} \, C_{ {\ell_0+1}}} > \sqrt{V_{{\ell_0}} \, C_{{\ell_0}}} + \sqrt{V_{\ell_0+1,\ell_0} \, C_{\ell_0+1,\ell_0}}.
\end{equation}
This analysis leads to a condition on the ratio \( \frac{V_{\ell_0+1,\ell_0}}{V_{\ell_0}} \), formalized in Proposition~\ref{prop1}. The result is established in a general setting where the single-level variance \( V_{\ell} \) may vary with the discretization level. Notably, no specific assumption was imposed on \( V_{\ell} \), allowing the proposition to accommodate a wide range of behavior. This generality is relevant in the proposed IS framework, where the single-level variance often decays to zero as the level increases. A simplified version of the result is presented in Corollary~\ref{cor}, corresponding to the standard case where the single-level variance is assumed to be uniformly bounded away from zero.

\begin{proposition}
\label{prop1}
Let \( \mathrm{TOL} \) be a prescribed target tolerance and \( L_{\mathrm{opt}} \) denote the corresponding optimal MLMC level, defined in~\eqref{lopt}, satisfying the bias constraint. If
\begin{enumerate}
\item  \( C_{\ell} = C_{0} 2^\ell \), and
\item \( C_{\ell+1,\ell} = C_{N_\ell} + C_{{\ell+1}} = 3C_{0}2^\ell \), 
\end{enumerate}
then the MLMC method is \textbf{strictly} better than the SLMC method in terms of sampling work, that is, 
\[\text{Work}_\text{sampling}^{\text{MLMC}} < \text{Work}_\text{sampling}^{\text{SLMC}},\]
\textbf{if and only if} a level \( \ell_0 < L_{\mathrm{opt}} \) exists such that
\begin{equation}
\label{condMLnewcase1}
\sqrt{\frac{V_{\ell+1,\ell}}{V_{{\ell}}}} < \frac{\sqrt{2\frac{V_{{\ell+1}}}{V_{{\ell}}}} - 1}{\sqrt{3}}  \qquad \textbf{for all } \quad \ell \geq \ell_0.
\end{equation}
\end{proposition}
\begin{proof}
See Appendix~\ref{proof1}.
\end{proof}
\begin{corollary}
\label{cor}
Let \( \mathrm{TOL} \) be a prescribed target tolerance and \( L_{\mathrm{opt}} \) be the corresponding optimal MLMC level defined in~\eqref{lopt}, satisfying the bias constraint. Let us suppose that:
\begin{enumerate}
    \item The variance of single-level estimators is nearly constant across levels, that is, \[
    1 - \varepsilon \leq \frac{V_0}{V_\ell} \leq 1 + \varepsilon \qquad \text{for all } \ell \geq 0,\text{ with } \varepsilon > 0.\]
\item  \( C_{\ell} = C_{0} 2^\ell \). 
\item \( C_{\ell+1,\ell} = C_{N_\ell} + C_{{\ell+1}} = 3C_{0}2^\ell \). 
\end{enumerate}
Thus, a \textbf{necessary} condition for the MLMC method to be \textbf{strictly better} than the SLMC method in terms of sampling work is that a level \( \ell_0 < L_{\mathrm{opt}} \) exists such that
 \begin{equation}
\label{corollaryeq}
\sqrt{\frac{V_{\ell_0+1,\ell_0}}{V_0}} < 
\frac{\sqrt{2(1+\varepsilon)} - \sqrt{1-\varepsilon}}{\sqrt{3(1 - \varepsilon^2)}}.
\end{equation}
\end{corollary}
\begin{proof}
See Appendix~\ref{proofcor}.
\end{proof}
\begin{remark}[Less Strict Necessary Conditions]
If we let \( \varepsilon \to 0 \) in~\eqref{corollaryeq}, corresponding to the case where the single-level variance \( V_{\ell} \) remains constant across levels, then~\eqref{corollaryeq} reduces to
\begin{equation}
\label{condMLnewcase1less}
\frac{V_{\ell_0+1,\ell_0}}{V_0} < \frac{(\sqrt{2} - 1)^2}{3} \approx 0.057.
\end{equation}
When the single-level variance is allowed to vary with the discretization level, as in Proposition~\ref{prop1}, a less stringent necessary condition for the MLMC method to outperform the SLMC method is that some level \( \ell_0 < L_{\mathrm{opt}} \) exists for which~\eqref{condMLnewcase1less} holds. Particularly, if the condition in~\eqref{condMLnewcase1less} is not met, the MLMC method offers no advantage over the SLMC method. Thus, the condition in~\eqref{condMLnewcase1less} serves as a practical preliminary check before verifying the stricter condition in~\eqref{condMLnewcase1}. Moreover, if the single-level variance satisfies
\begin{equation}
\frac{V_{{\ell+1}}}{V_{{\ell}}} < \frac{1}{\sqrt{2}} \quad \text{for all } \ell < L_{\mathrm{opt}},
\end{equation}
then the condition in~\eqref{condMLnewcase1} fails for all choices of \( \ell_0 \), and the MLMC method does not offer any computational benefits. This statement is intuitive. The MLMC method exploits variance reduction across levels, but if the variance already decays rapidly at the single level, the additional overhead of a multilevel structure yields no meaningful gain.
\end{remark}

From Propositions~\ref{prop1}, we conclude that the coarse level \( \ell_0 \) must satisfy certain conditions for the MLMC method to outperform the SLMC method. The first level \( \ell_0 \) for which the condition is satisfied is considered optimal and is denoted by \( \ell_0^{\mathrm{opt}} \). Any level \( \ell_0 > \ell_0^{\mathrm{opt}} \) where the condition also holds is not considered optimal because increasing \( N_{\ell_0} \) further results in unnecessary computation and higher sampling cost. If the condition is not satisfied for any \( \ell_0 < L_{\mathrm{opt}} \), this approach sets \( \ell_0^{\mathrm{opt}} = L_{\mathrm{opt}} \), in which case the MLMC method reduces to the SLMC method. Algorithm~\ref{alg:N0opt} computes the optimal coarse level \( \ell_0^{\mathrm{opt}} =\ell_0^{\mathrm{opt}}( \mathrm{TOL} )\) in the general case where the single-level variance may decay with increasing levels.
\begin{algorithm}
\caption{Determining \( \ell_0^{\mathrm{opt}} \)}
\label{alg:N0opt}
\begin{algorithmic}[1]

\State \textbf{Input:} Initial coarse level \( N_0 \), target tolerance \( \mathrm{TOL} \)

\State Compute \( N_{\mathrm{opt}} \) using~\eqref{eq:Nopt}
\State \( L_{\mathrm{opt}} \gets \log_2 \left( \frac{N_{\mathrm{opt}}}{N_0} \right) \)
\State  \( \ell_0^{\mathrm{opt}} \gets 0\)
\For{ \( \ell = 0 \) to \( L_{\mathrm{opt}}-1 \) }
\If {\(
\sqrt{V_{{\ell+1}} C_{{\ell+1}}} < \sqrt{V_{\ell} C_{\ell}} + \sqrt{V_{\ell+1, \ell} C_{\ell+1, \ell}} 
\) }
    \State \( \ell_0^{\mathrm{opt}} \gets \ell +1\)
\EndIf

\EndFor
\State \textbf{Return:} \( \ell_0^{\mathrm{opt}} \)
\end{algorithmic}
\end{algorithm}
\section{Multilevel Importance Sampling } 
\label{Section 4}
When estimating rare-event probabilities, the standard MLMC method becomes inefficient due to the high variance \( V_{\ell,\ell-1} \) at each level, leading to substantial computational costs. Compared to SLIS, the MLMC method may require significantly more samples because the level variances \( V_{\ell,\ell-1} \) are often much larger than the variance \( V_N^{\mathrm{IS}} \) achieved using SLIS. For example, in the numerical experiment presented in Section~\ref{section num}, the observed values of \( V_{\ell,\ell-1} \) lie in the range \( [10^{-3}, 10^{-5}] \) and are relatively high compared to the much smaller variances, illustrated in Figure~\ref{varianceSLIS}, obtained using SLIS. This study addresses this problem by combining IS with the multilevel estimator, resulting in the MLIS method. The aim is to reduce the variance \( V_{\ell,\ell-1} \) at each level by applying the same change of measure, introduced in Section~\ref{Section 2}, to levels \( \ell \) and \( \ell - 1 \).  

The same complexity analysis for the MLMC method extends to MLIS, with appropriate substitutions. The single-level variance \( V_{\ell} \) is replaced by the IS variance \( V_{\ell}^{\mathrm{IS}} \) (corresponding to \( V_{\ell}^{\mathrm{IS}} \) in~\eqref{varianceSLIS}), and the level-difference variance \( V_{\ell,\ell-1} \) is replaced by the MLIS variance \( V^{\mathrm{IS}}_{\ell,\ell-1} \). Similarly, the cost per sample \( C_{\ell} \) is replaced by \( C_{\ell}^{\mathrm{IS}} \), which represents the cost of generating a single IS sample with \( N_\ell \) time steps, and \( C_{\ell,\ell-1} \) is replaced by the corresponding MLIS cost \( C^{\mathrm{IS}}_{\ell,\ell-1} \). Algorithm~\ref{alg:N0opt} can still be applied to determine the optimal coarse level \( \ell_0^{\mathrm{opt}} \) for MLIS. However, the accuracy of the HJB-PDE \( \epsilon_{\mathrm{PDE}} \) must also be input because the variances in the algorithm depend on the accuracy of the control, which is determined by the precision with which the associated auxiliary HJB-PDE is solved.

The following section proposes an MLIS estimator based on a single-level likelihood, considering two illustrative cases: using nonsmooth functions \( g_w \) and \( f \) and using their smoothed versions. In both cases, the condition in~\eqref{condMLnewcase1less} is not satisfied, and MLIS effectively reduces to SLIS, prompting the use of a {common likelihood} across levels, a novel approach to our knowledge. In the smoothed case, when employing a common likelihood, the condition in~\eqref{condMLnewcase1} is satisfied, revealing a clear improvement of MLIS over SLIS. These examples offer a numerical validation of the conclusions in Proposition~\ref{prop1}.
\subsection{Single-Level Likelihood}
A natural extension of the IS approach from the single-level setting to the multilevel framework is to apply the same single-level change of measure, as introduced in Section~\ref{Section 2}, to both levels \( \ell \) and level \( \ell - 1 \). Such an approach has previously been proposed in the context of combining optimal IS with the MLMC method for McKean--Vlasov SDEs~\cite{ben2025multilevel}. The resulting MLIS estimator is defined as follows:
\begin{equation}
\label{eq:MLIS}
\mathcal{A}_{\mathrm{MLIS}} = \frac{1}{M_0} \sum_{m=1}^{M_0} \tilde{g}_w^{(0,m)} + \sum_{\ell=1}^{L} \frac{1}{M_\ell} \sum_{m=1}^{M_\ell} \left( \tilde{g}_w^{\ell} - \tilde{g}_w^{\ell-1} \right)^{(m)}, \qquad \text{where} \quad  \tilde{g}_w^{\ell} = g_w\left(\tilde{Z}_{N_\ell}^{\ell}\right) L_{N_\ell}.
\end{equation}
Moreover, \( \tilde{Z}_{N_\ell}^{\ell}\) and \( \tilde{Z}_{N_{\ell-1}}^{\ell-1}\) are obtained by solving the following controlled Euler--Maruyama systems at levels \( \ell \) and \( \ell-1 \), respectively:
\begin{align}
\tilde{\boldsymbol{X}}_{n+1}^{\ell} &= \tilde{\boldsymbol{X}}_n^{\ell} + a(t_n^\ell, \tilde{\boldsymbol{X}}_n^{\ell}) \Delta t_\ell + b(t_n^\ell, \tilde{\boldsymbol{X}}_n^{\ell}) \left( \Delta W_n^\ell + \zeta^*(t_n^\ell, \tilde{\boldsymbol{X}}_n^{\ell}, \tilde{Z}_n^{\ell}) \Delta t_\ell \right), \\
\tilde{Z}_{n+1}^{\ell} &= \tilde{Z}_n^{\ell} + f(h(\tilde{\boldsymbol{X}}_n^{\ell})) \Delta t_\ell, \qquad 
\Delta t_\ell = \frac{T}{N_\ell}, \qquad 
t_n^\ell = n \Delta t_\ell, \quad 
n = 0, \dots, N_\ell - 1,
\label{Zldiff}
\end{align}
and
\begin{align}
\tilde{\boldsymbol{X}}_{n+1}^{\ell-1} &= \tilde{\boldsymbol{X}}_n^{\ell-1} + a(t_n^{\ell-1}, \tilde{\boldsymbol{X}}_n^{\ell-1}) \Delta t_{\ell-1} + b(t_n^{\ell-1}, \tilde{\boldsymbol{X}}_n^{\ell-1}) \left( \Delta W_n^{\ell-1} + \zeta^*(t_n^{\ell-1}, \tilde{\boldsymbol{X}}_n^{\ell-1}, \tilde{Z}_n^{\ell-1}) \Delta t_{\ell-1} \right), \\
\tilde{Z}_{n+1}^{\ell-1} &= \tilde{Z}_n^{\ell-1} + f(h(\tilde{\boldsymbol{X}}_n^{\ell-1})) \Delta t_{\ell-1}, \quad 
\Delta t_{\ell-1} = \frac{T}{N_{\ell-1}}, \quad 
t_n^{\ell-1} = n \Delta t_{\ell-1}, \quad 
n = 0, \dots, N_{\ell-1} - 1.
\label{Zl-1diff}
\end{align}
This work defines the coarse-level increment as the sum of two fine-level increments to couple the Brownian increments between levels \( \ell \) and \( \ell - 1 \):
\begin{equation}
\Delta W_n^{\ell-1} = \Delta W_{2n}^{\ell} + \Delta W_{2n+1}^{\ell}, \qquad n = 0, \dots, N_{\ell-1} - 1.
\end{equation}
Finally, the likelihood terms \( L_{N_{\ell}} \) and \( L_{N_{\ell-1}} \) are different, and are defined as
\begin{equation}
\label{likelihood1}
L_{N_{\ell}} = \prod_{n=0}^{N_\ell - 1} \exp \left\{
  -\frac{1}{2} \Delta t_\ell \left\| \zeta^*(t_n^\ell, \tilde{\boldsymbol{X}}_n^{\ell}, \tilde{Z}_n^{\ell}) \right\|^2 
  - \left\langle \Delta W_n^{\ell}, \zeta^*(t_n^\ell, \tilde{\boldsymbol{X}}_n^{\ell}, \tilde{Z}_n^{\ell}) \right\rangle
\right\},
\end{equation}

\begin{equation}
\label{likelihood2}
L_{N_{\ell-1}} = \prod_{n=0}^{N_{\ell-1} - 1} \exp \left\{
  -\frac{1}{2} \Delta t_{\ell-1} \left\| \zeta^*(t_n^{\ell-1}, \tilde{\boldsymbol{X}}_n^{\ell-1}, \tilde{Z}_n^{\ell-1}) \right\|^2 
  - \left\langle \Delta W_n^{\ell-1}, \zeta^*(t_n^{\ell-1}, \tilde{\boldsymbol{X}}_n^{\ell-1}, \tilde{Z}_n^{\ell-1}) \right\rangle
\right\}.
\end{equation}
\subsubsection{Nonsmooth Case}
The same application described in Section~\ref{section num} is considered, where the observable $g_w$ and drift $f$ are nonsmooth functions. This work investigates the performance of the proposed MLIS estimator \eqref{eq:MLIS} by fixing an arbitrary coarse level \( N_0 = 20 \) and computing the variance \( V^{\mathrm{IS}}_{\ell,\ell-1} \) for values of \( \ell \) to assess whether the condition in~\eqref{condMLnewcase1less} can be satisfied. This analysis is repeated for various levels of the auxiliary HJB-PDE accuracy, following the same approach as earlier to study the behavior of \( V_N^{\mathrm{IS}} \) (see Section~\ref{section num}).
\begin{figure}[ht]
\begin{center}
\includegraphics[width=2.4in]{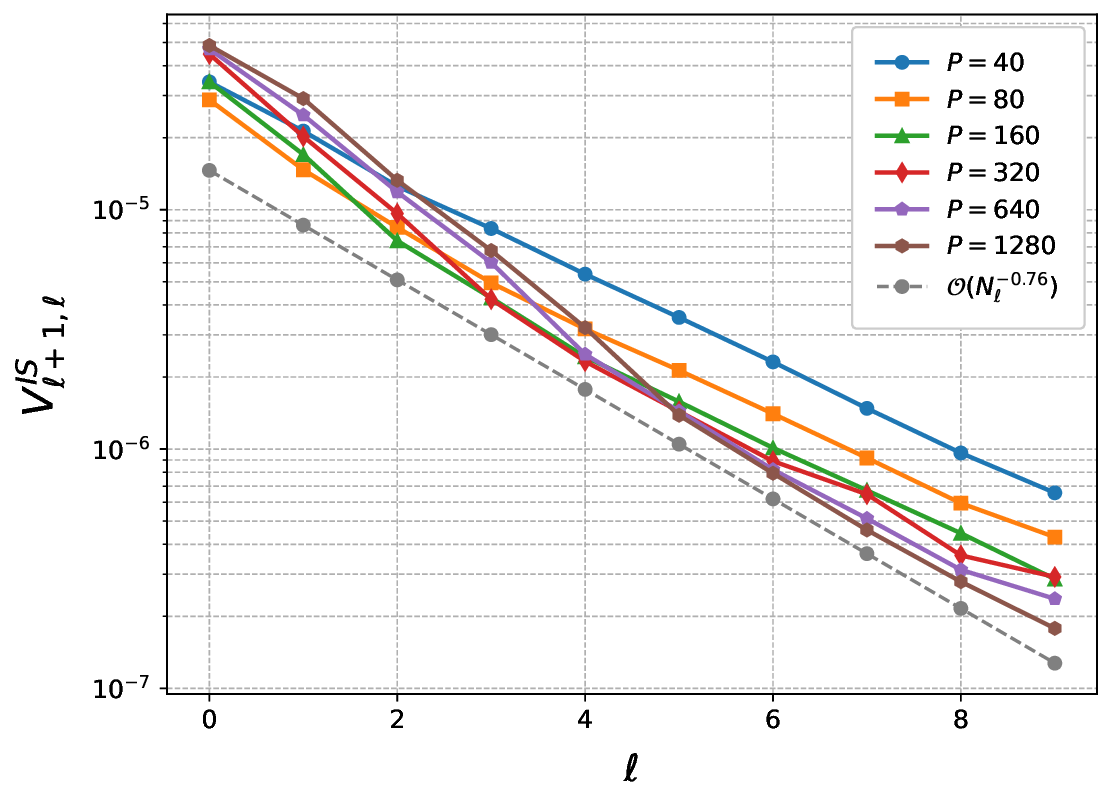}
\hspace{-0.1cm}
\includegraphics[width=2.4in]{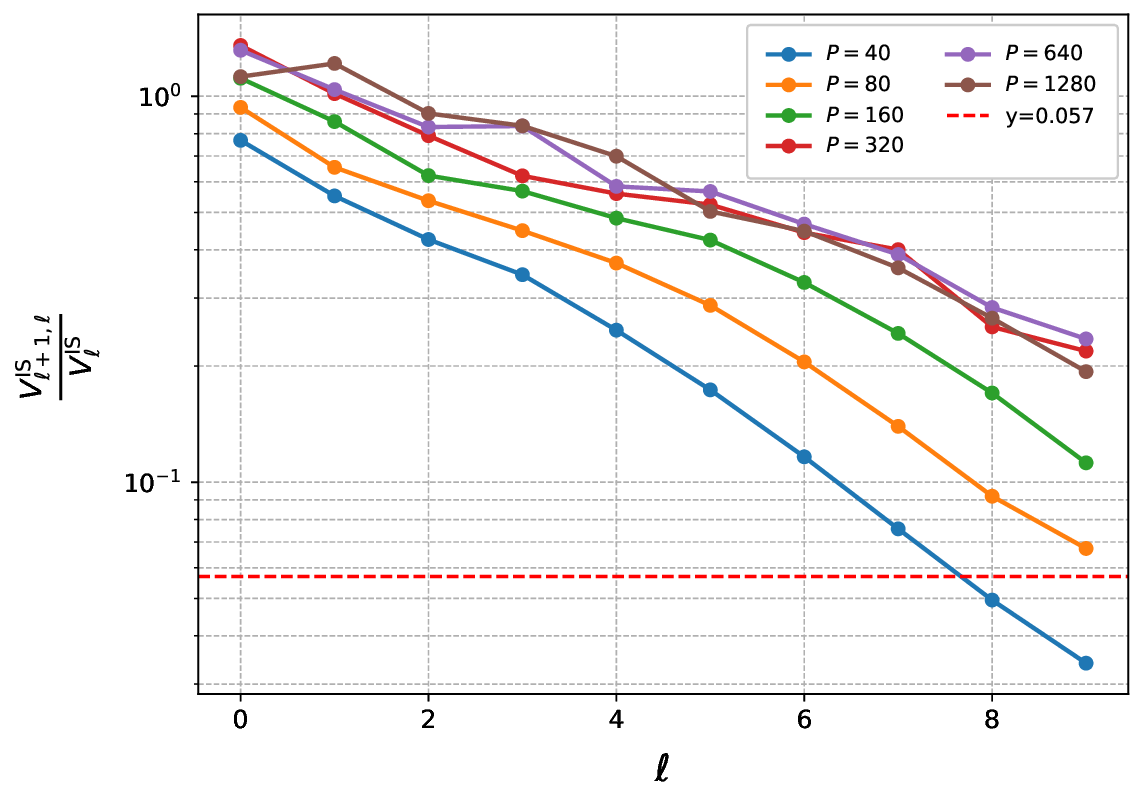}
\caption{Nonsmooth case. Left: Variance \( V^{\mathrm{IS}}_{\ell+1,\ell} \) for various PDE accuracies. 
Right: Ratio \( \frac{V^{\mathrm{IS}}_{\ell+1,\ell}}{V_{{\ell}}^{\mathrm{IS}} }\) for different PDE accuracies.}
\label{example1}
\vspace{-3mm}
\end{center}
\end{figure}

 Figure~\ref{example1} presents the results and reveals that no prospect exists for MLIS to outperform SLIS when \( P\geq 80 \). The only scenario in which MLIS may offer a slight advantage is when \( P = 40 \). Even then, the advantage is only for tolerances requiring numerous levels, specifically, when \( \mathrm{TOL} \leq 0.003 \), corresponding to \( L_{\mathrm{opt}} \geq 8 \).
We compute the sampling work of MLIS using the arbitrary choice \( N_0 = 20 \) to further investigate this observation. Over the range of relative tolerances \( [0.1,0.003] \), SLIS consistently outperforms MLIS for all auxiliary HJB-PDE accuracy values. For a more rigorous comparison, the optimal coarse level \( \ell_0^{\mathrm{opt}} \) was computed using Algorithm~\ref{alg:N0opt}, demonstrating that \( \ell_0^{\mathrm{opt}} = L_{\mathrm{opt}} \) for all considered tolerance values, implying that MLIS reduces to SLIS across the entire tolerance range. 

The inefficiency of MLIS in this application stems from the slow decay of the variance \( V^{\mathrm{IS}}_{\ell,\ell-1} \), which behaves as \( \mathcal{O}(N_\ell^{-0.76}) \), as presented in Figure~\ref{example1}. This limited decay is primarily due to the lack of Lipschitz continuity in the observable \( g_w \) and drift \( f \) functions. A natural method to improve the ratio \( \frac{V^{\mathrm{IS}}_{\ell+1,\ell} }{V_{{\ell}}^{\mathrm{IS}}} \) is to enhance the decay rate of \( V^{\mathrm{IS}}_{\ell+1,\ell} \) by smoothing these functions. The next example explores this smoothing approach.

\subsubsection{Smooth Case}
\label{smooth section}
This section introduces the smooth approximations \( f^d \) and \( g_w^c \) for the functions \( f \) and \( g_w \), respectively, to improve the multilevel variance convergence rates. Rather than estimating the original quantity \( q_w \), the smoothed quantity $q_w^{c,d} =\mathbb{E}\left[g_w^c(Z^d(T))\right]$ is approximated, where $Z^d(T) = \int_0^T f^d\left( h(\boldsymbol{X}(s)) \right) \, ds$. 

Ideally, the smoothing parameters \(c\) and \(d\)
should be chosen as a function of \(\mathrm{TOL}\) in order to retain the effective
unbiasedness of the estimator for \(q_{w}\). The construction of such a
tolerance dependent strategy is developed in Section~\ref{errsplit}. In the following, however, we fix the smoothing parameters to
\(c = 0.5\) and \(d = 0.122\) for the entire range of \(\mathrm{TOL}\). This
choice will be justified in Section~\ref{choicepar}, where it is shown to
yield a negligible smoothing error. Consequently, the analysis of the
computational work of the MLIS method as a function of \(\mathrm{TOL}\) in Figure \ref{works} and \ref{SLvsML} is performed without the additional complication of a tolerance dependent
smoothing error.

The smoothing functions \( f^d \) and \( g_w^c \) are defined as follows:
\begin{equation}
\label{fdgc}
\begin{aligned}
f^d(x) = 1 - 
\begin{cases}
0 & \text{if } x - \gamma_{\text{th}} \geq d, \\
0.5 + \dfrac{x - \gamma_{\text{th}}}{2d} & \text{if } -d < x - \gamma_{\text{th}} < d, \\
1 & \text{if } x - \gamma_{\text{th}} \leq -d,
\end{cases}
\quad
g_w^c(x) =
\begin{cases}
0 & \text{if } x \leq w - c, \\ 0.5 + \dfrac{x - w}{2c} & \text{if } w - c < x < w + c, \\ 1 & \text{if } x \geq w + c.
\end{cases}
\end{aligned}
\end{equation}
As the SDE has constant diffusion, and the drift and observable functions are Lipschitz continuous, the variance of the MLMC level differences \( V_{\ell+1,\ell} \) (without IS) are expected to decay at rate \( \mathcal{O}(N_\ell^{-2}) \), that is, with an order of 2. Figure~\ref{ratesVMLMC} confirms this theoretical rate numerically.

\begin{figure}[ht] 
\begin{center}
\includegraphics[scale = 0.4]{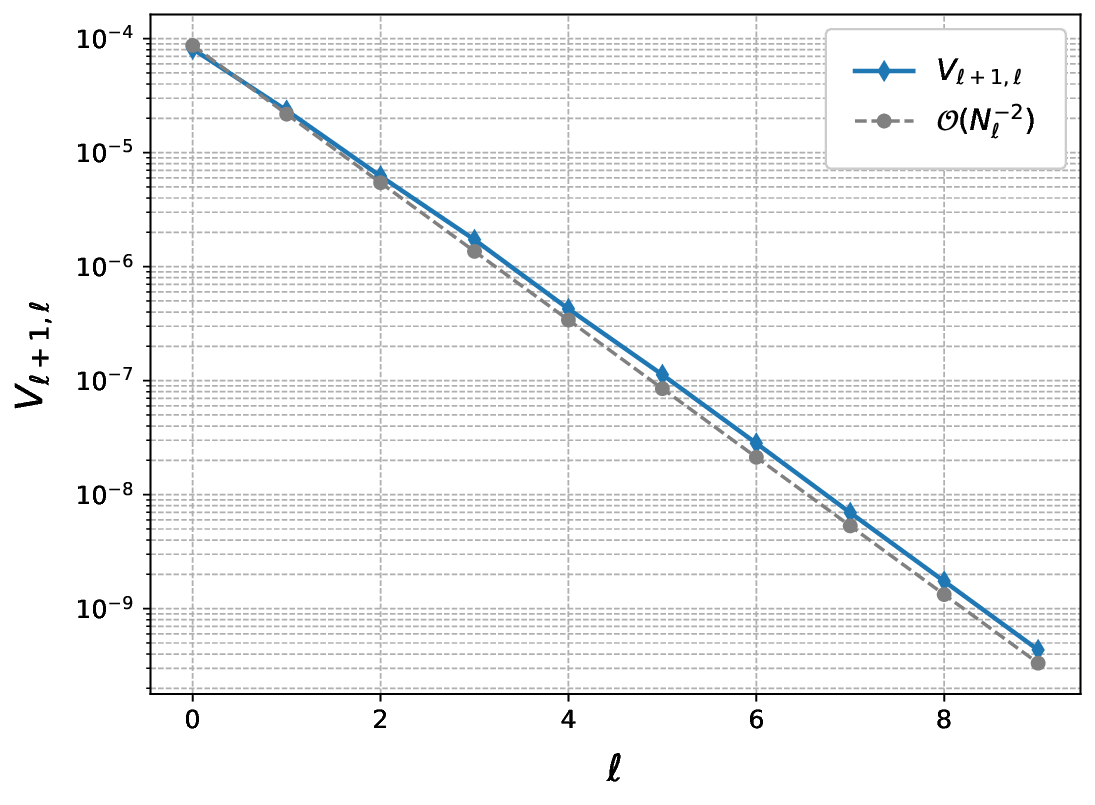} 
\caption{Variance decay of the crude multilevel Monte Carlo (MLMC) method without importance sampling (IS) in the smooth case.}
\label{ratesVMLMC}
\vspace{-3mm}
\end{center}
\end{figure} 
We now consider the smoothed dynamics obtained by replacing $f$ with $f^{d}$ in
\eqref{Zldiff}--\eqref{Zl-1diff}. We denote the corresponding auxiliary
processes by $\tilde{Z}^{\ell,d}$ and $\tilde{Z}^{\ell-1,d}$. In the smoothed
setting, the likelihood weights defined in \eqref{likelihood1}--\eqref{likelihood2}
also depend on the parameters $(c,d)$, since both the control $\zeta^{*}$ and the process $\tilde{Z}$ are replaced 
by their smoothed counterparts. We therefore
write them as $L_{N_\ell}^{c,d}$ and $L_{N_{\ell-1}}^{c,d}$ in what follows. The MLIS estimator~\eqref{eq:MLIS} employs {different likelihoods} for levels~$\ell$ and $\ell - 1$. Therefore, the variance of level differences in MLIS in the smooth case takes the following form:
\begin{equation}
V^{\mathrm{IS}}_{\ell,\ell-1}(c,d)
= \operatorname{Var} \left[ 
g_w^c\left(\tilde{Z}^{\ell,d}_{N_\ell}\right) L_{N_\ell}^{c,d}
- g_w^c\left(\tilde{Z}^{\ell-1,d}_{N_{\ell-1}}\right) L_{N_{\ell-1}}^{c,d}
\right].
\end{equation}
Based on the numerical results in Figure~\ref{example2difflik}, this formulation yields a variance decay rate for MLIS of approximately an order of~1. Thus, the ratio \( \frac{V^{\mathrm{IS}}_{\ell+1,\ell} }{V_{{\ell}}^{\mathrm{IS}} }\) is not significantly reduced. This deterioration in the variance decay rate, compared to the standard MLMC method, is primarily due to different likelihood functions across levels. A natural solution to recover the optimal decay rate (an order of 2) is to apply the same likelihood at both levels. 
\begin{figure}[ht]
\begin{center}
\includegraphics[width=2.4in]{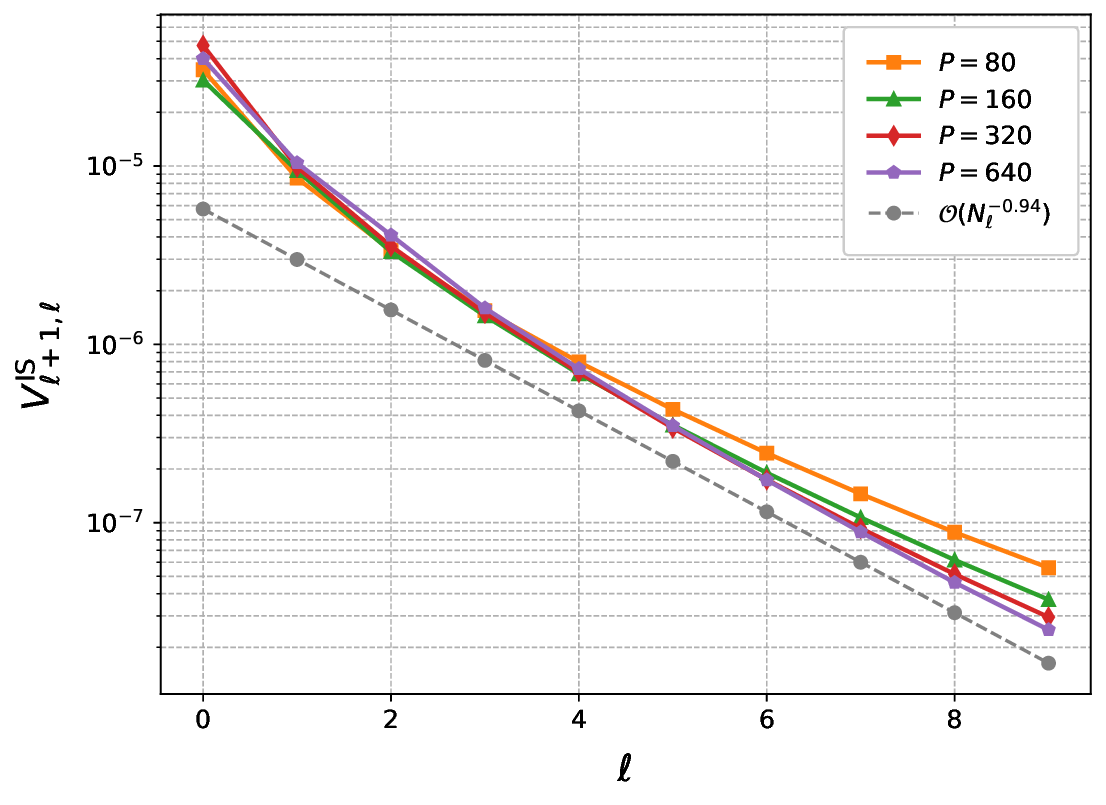}
\hspace{-0.1cm}
\includegraphics[width=2.4in]{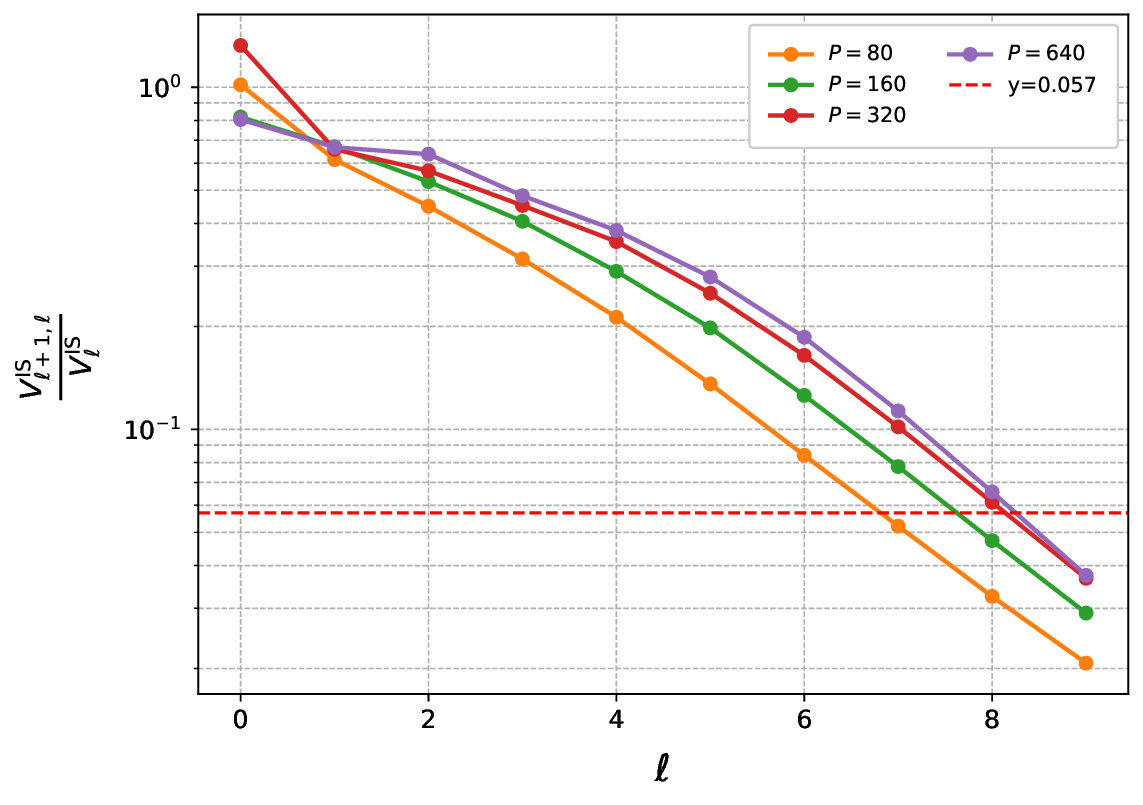}
\caption{Smooth case ($d = 0.122$ and $c = 0.5$) with single-level likelihood. Left: Variance \( V^{\mathrm{IS}}_{\ell+1,\ell} \). Right: Ratio \( \frac{V^{\mathrm{IS}}_{\ell+1,\ell} }{ V_{{\ell}}^{\mathrm{IS}}} \).}
\label{example2difflik}
\vspace{-3mm}
\end{center}
\end{figure}
\subsection{Common Likelihood}
\label{common likelihood}
Next, this work proposes a novel variant of the MLIS framework that slightly departs from the standard approach. Instead of computing the optimal control independently at levels \( \ell \) and \( \ell - 1\), this approach computes the control only at level \( \ell \) and reuses it at level \( \ell - 1\). In other words, the controlled process \( \tilde{Z}_{N_\ell}^{\ell} \) follows the same dynamics as those previously described in~\eqref{Zldiff}, and the coarse path \( \tilde{Z}_{N_{\ell-1}}^{\ell-1} \) is defined using a coupled construction.
The shifted Brownian increment at level \( \ell \) is denoted by \( \Delta \tilde{W}_n^\ell = \Delta W_n^\ell + \zeta^*(t_n^\ell, \tilde{\boldsymbol{X}}_n^\ell, \tilde{Z}_n^\ell)\, \Delta t_\ell \) and is a Brownian increment under a new measure by Girsanov’s theorem~\cite{oksendal2003stochastic}. Then, the coarse-level increment is defined as the sum of two consecutive fine-level increments:\[
\Delta \tilde{W}_n^{\ell-1} = \Delta \tilde{W}_{2n}^{\ell} + \Delta \tilde{W}_{2n+1}^{\ell}, \qquad n = 0, \dots, N_{\ell-1} - 1.\]
Consequently, the coarse-level controlled process \( \tilde{Z}_{N_{\ell-1}}^{\ell-1} \) evolves according to the following:
\begin{align}
\tilde{\boldsymbol{X}}_{n+1}^{\ell-1} &= \tilde{\boldsymbol{X}}_n^{\ell-1} + a(t_n^{\ell-1}, \tilde{\boldsymbol{X}}_n^{\ell-1}) \Delta t_{\ell-1} + b(t_n^{\ell-1}, \tilde{\boldsymbol{X}}_n^{\ell-1}) \left( \Delta W_n^{\ell-1} + \tilde{\zeta}(t_n^{\ell-1}, \tilde{\boldsymbol{X}}_n^{\ell-1}, \tilde{Z}_n^{\ell-1}) \Delta t_{\ell-1} \right), \\
\tilde{Z}_{n+1}^{\ell-1} &= \tilde{Z}_n^{\ell-1} + f(h(\tilde{\boldsymbol{X}}_n^{\ell-1})) \Delta t_{\ell-1}, \quad 
\Delta t_{\ell-1} = \frac{T}{N_{\ell-1}}, \quad 
t_n^{\ell-1} = n \Delta t_{\ell-1},
\end{align}
where the coarse-level control is defined by
\begin{equation}
\tilde{\zeta}(t_n^{\ell-1}, \tilde{\boldsymbol{X}}_n^{\ell-1}, \tilde{Z}_n^{\ell-1}) = \frac{1}{2} \left[ \zeta^*(t_{2n}^{\ell}, \tilde{\boldsymbol{X}}_{2n}^{\ell}, \tilde{Z}_{2n}^{\ell}) + \zeta^*\left(t_{2n+1}^{\ell}, \tilde{\boldsymbol{X}}_{2n+1}^{\ell}, \tilde{Z}_{2n+1}^{\ell}\right) \right], \quad n = 0, \dots, N_{\ell-1} - 1.
\end{equation}
Unlike the single-level likelihood approach, where the fine and coarse paths are coupled prior to the change of measure, the present method performs the coupling after the change of measure. This strategy ensures that the same likelihood \( L_{N_\ell} \) is employed for levels \( \ell \) and \( \ell-1\). Under this construction, the variance of the multilevel difference in the smooth case becomes the following:
\begin{equation}
\tilde{V}^{\mathrm{IS}}_{\ell,\ell-1} (c,d)= \operatorname{Var} \left[ \left( g_w^c\left(\tilde{Z}_{N_\ell}^{\ell,d}\right) - g_w^c\left(\tilde{Z}_{N_{\ell-1}}^{\ell-1,d}\right) \right) L_{N_\ell}^{c,d} \right].
\end{equation}
This variance is expected to decay at the same rate (rate = 2) as the standard MLMC variance.
\begin{figure}[ht]
\begin{center}
\includegraphics[width=2.4in]{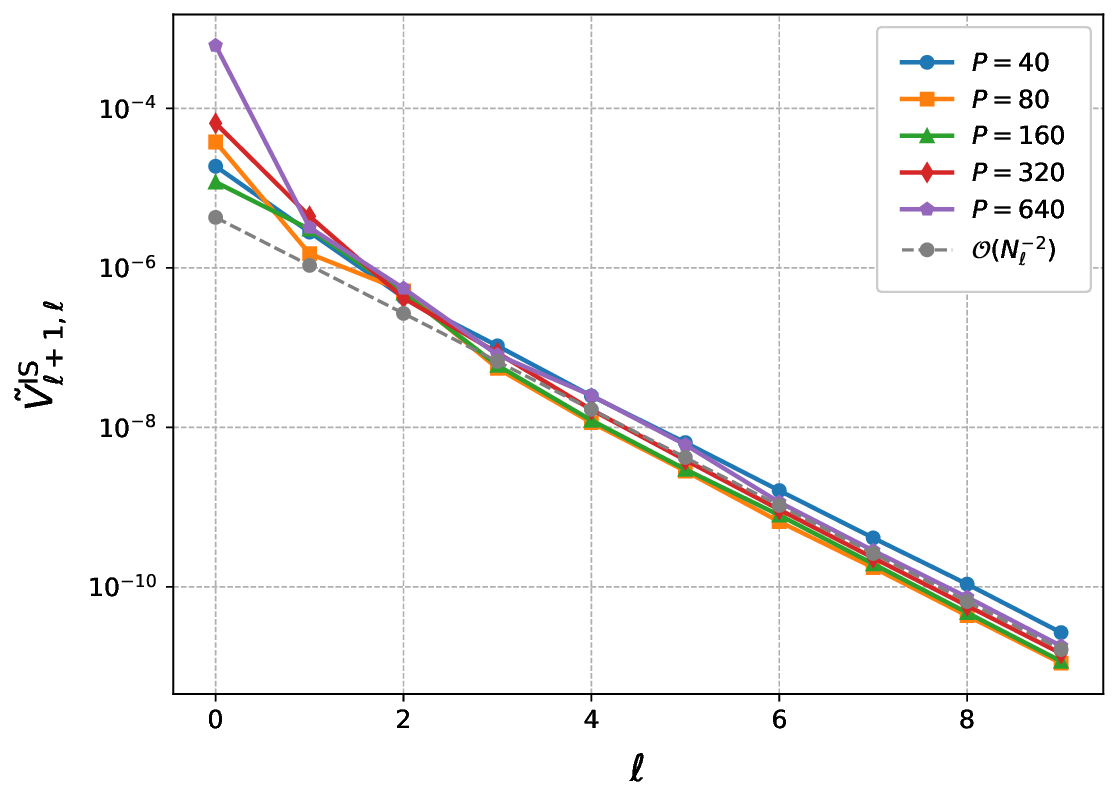}
\hspace{-0.1cm}
\includegraphics[width=2.4in]{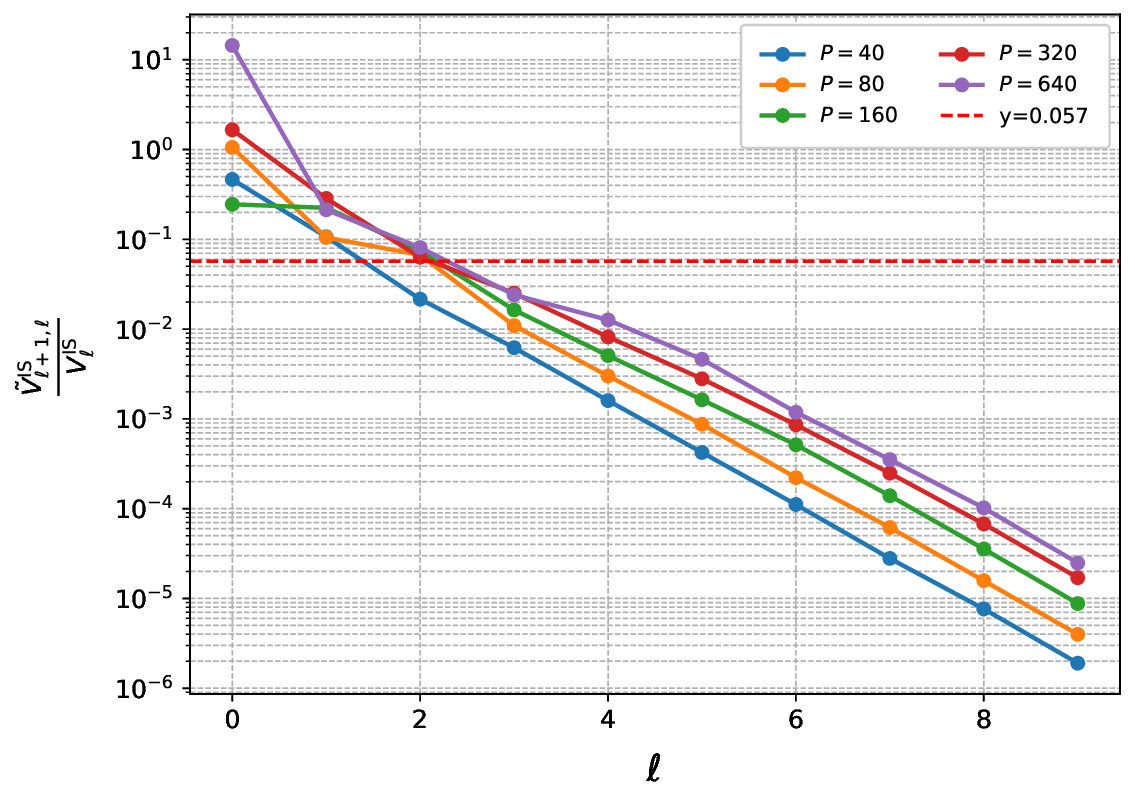}
\caption{Smooth case ($d = 0.122$ and $c = 0.5$) using the common likelihood. Left: Variance \( \tilde{V}^{\mathrm{IS}}_{\ell+1,\ell} \). Right: Ratio \( \frac{\tilde{V}^{\mathrm{IS}}_{\ell+1,\ell} }{ V_{{\ell}}^{\mathrm{IS}} }\).}
\label{example2samelik}
\vspace{-10mm}
\end{center}
\end{figure}
\begin{figure}[ht] 
\begin{center}
\vspace{-3mm}
\includegraphics[scale = 0.55]{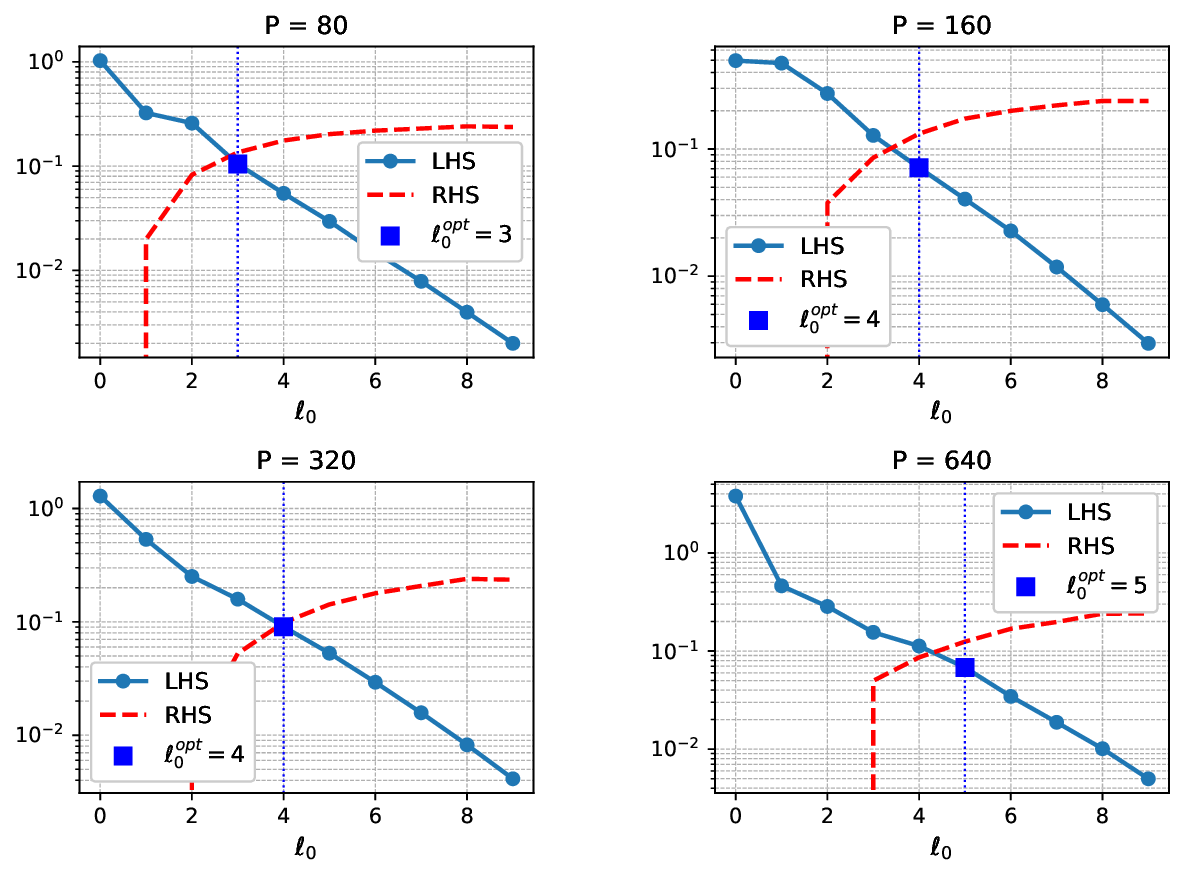} 
\caption{Left-hand side $\mathrm{(LHS)}=\sqrt{\frac{\tilde{V}^{\mathrm{IS}}_{\ell_0+1,\ell_0}}{\tilde{V}^{\mathrm{IS}}_{{\ell_0}}}}$ and right-hand side $\mathrm{(RHS)}=\frac{\sqrt{2\frac{\tilde{V}^{\mathrm{IS}}_{{\ell_0+1}}}{\tilde{V}^{\mathrm{IS}}_{{\ell_0}}}} - 1}{\sqrt{3}} $.}
\label{example2subplot}
\end{center}
\end{figure} 
Figure~\ref{example2samelik} plots \( \tilde{V}^{\mathrm{IS}}_{\ell+1,\ell} \) along with the ratio \( \frac{\tilde{V}^{\mathrm{IS}}_{\ell+1,\ell} }{ V_{{\ell}}^{\mathrm{IS}}} \). When using a common likelihood, the decay rate of \( \tilde{V}^{\mathrm{IS}}_{\ell+1,\ell} \) improves to an order of~2, and the necessary condition~\eqref{condMLnewcase1less} is satisfied for smaller values of~$\ell_0$ compared with the single-level likelihood case. This result suggests that MLIS can outperform SLIS for reasonably small tolerances.

The condition \eqref{condMLnewcase1less} is necessary but not sufficient; hence, this work proceeds to verify the condition in~\eqref{condMLnewcase1} to determine whether the improved variance decay rate is sufficient for MLIS to outperform SLIS. Figure~\ref{example2subplot} plots both sides of the inequality~\eqref{condMLnewcase1} to identify the values of $\ell_0$ for which the condition is satisfied. 
The sufficient condition is satisfied for the following:
\begin{itemize}
    \item \( \ell_0 \geq 3 \) for \( P = 80 \),
    \item \( \ell_0 \geq 4 \) for \( P = 160 \) and \(320 \),
    \item \( \ell_0 \geq 5 \) for \( P = 640 \).
\end{itemize}
Let \( \mathrm{TOL}_{\text{lim}}^{(P)} \) denote the minimum tolerance below which MLIS is expected to outperform SLIS for a given HJB--PDE resolution \( P \). This threshold is derived from the corresponding \( L_{\mathrm{opt}} \), computed using \eqref{lopt} and~\eqref{eq:Nopt}. As MLIS can outperform SLIS only if $\ell_0<L_{\mathrm{opt}}$, the minimum tolerances should satisfy the following constraints:
\begin{itemize}
    \item \( \mathrm{TOL}_{\text{lim}}^{(80)} = 0.03 \) corresponding to \( L_{\mathrm{opt}} = 4 \),
    \item \( \mathrm{TOL}_{\text{lim}}^{(160)} = \mathrm{TOL}_{\text{lim}}^{(320)} = 0.016 \), corresponding to \( L_{\mathrm{opt}} = 5 \),
    \item \( \mathrm{TOL}_{\text{lim}}^{(640)} = 0.008 \) corresponding to \( L_{\mathrm{opt}} = 6 \).
\end{itemize}
These results were numerically verified by plotting the sampling work of SLIS and MLIS in the smooth case for various auxiliary HJB-PDE accuracy values in Figure~\ref{works}. As expected, MLIS requires less sampling work than SLIS for tolerances below the corresponding threshold \( \mathrm{TOL}_{\text{lim}}^{(P)} \) for each auxiliary HJB-PDE resolution. These results confirm Proposition~\ref{prop1}: the necessary and sufficient condition in~\eqref{condMLnewcase1} accurately predicts when MLIS outperforms SLIS. For each auxiliary HJB-PDE accuracy value, MLIS becomes more efficient once the tolerance drops below \( \mathrm{TOL}_{\text{lim}}^{(P)} \).
\begin{figure}[ht]
\begin{center}
\includegraphics[scale = 0.55]{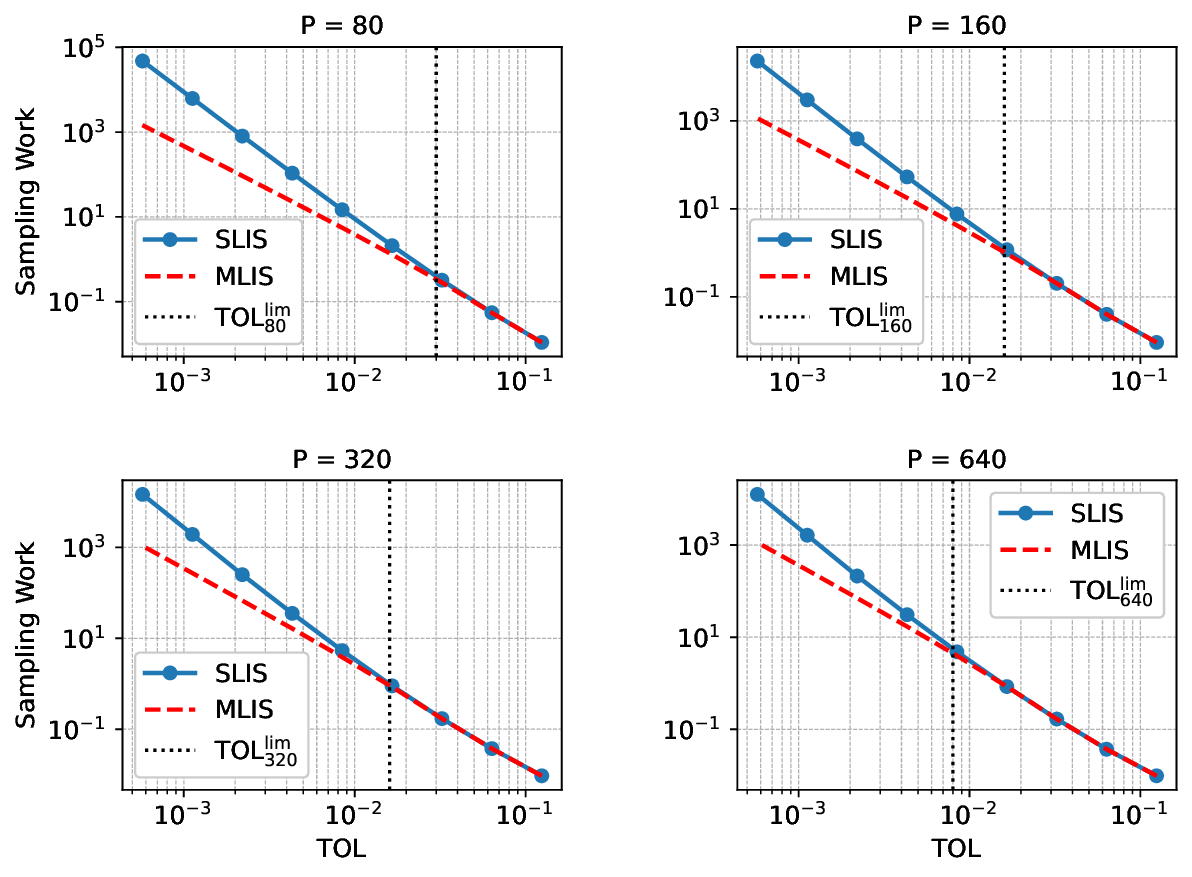}
\hspace{-0.1cm}
\caption{Sampling work for single-level importance sampling (SLIS) vs. multilevel importance sampling (MLIS) by auxiliary HJB-PDE accuracy.}
\label{works}
\vspace{-3mm}
\end{center}
\end{figure}

Notably, the work in Figure~\ref{works} reflects only the sampling cost. In this experiment, the auxiliary HJB-PDE accuracy was fixed across all tolerances, and the sampling work was analyzed as a function of \(\mathrm{TOL}\). However, as demonstrated in Section~\ref{section auxiliary HJB}, accounting for the auxiliary HJB-PDE accuracy is crucial in the single-level case, where an optimal auxiliary HJB-PDE tolerance exists for each \(\mathrm{TOL}\). A similar investigation is necessary in the multilevel setting, where the total computational cost considering the auxiliary HJB-PDE and sampling costs should be assessed. The next section addresses this full work analysis.

\section{Incorporating Auxiliary HJB-PDE Accuracy in MLIS Efficiency}
\label{Section 5}
\subsection{MLIS Optimal Work}
\label{NumEx}
The total computational cost, including sampling and auxiliary HJB-PDE costs, must be considered to evaluate the true efficiency of MLIS compared to SLIS. For MLIS, this work assumes that all levels share the same auxiliary HJB-PDE accuracy, denoted by \( \epsilon_{\mathrm{PDE}} \). The total work is defined as follows:
\begin{equation}
\label{totalWorkMLIS}
\text{Work}_{\mathrm{MLIS}}(\epsilon_{\mathrm{PDE}}) = \text{Work}_\text{sampling}^{\mathrm{MLIS}}(\epsilon_{\mathrm{PDE}}) + \text{Work}_\text{PDE}(\epsilon_{\mathrm{PDE}}),
\end{equation}
where 
\begin{equation}
        \label{MLISwork}
        \text{Work}_\text{sampling}^{\mathrm{MLIS}} = \left( \frac{2 C}{q_w \, \mathrm{TOL}} \right)^2
\left( \sqrt{V_{ {\ell_0^\mathrm{opt}} }^{\mathrm{IS}}C_{{\ell_0^\mathrm{opt}}}^{\mathrm{IS}}} 
        + \sum_{\ell = \ell_0^{\mathrm{opt}} + 1}^{L_{\mathrm{opt}}} \sqrt{\tilde{V}^{\mathrm{IS}}_{\ell,\ell-1} C^{\mathrm{IS}}_{\ell,\ell-1}} \right)^2.
 \end{equation}
The optimal auxiliary HJB-PDE accuracy was selected to minimize this total cost. In the comparison strategy, the total computational costs for SLIS and MLIS are optimized independently over the auxiliary HJB-PDE accuracy \( \epsilon_{\mathrm{PDE}} \), and their respective optimal configurations are compared. This approach ensures a fair and consistent assessment of the overall efficiency of both methods.

For a fixed relative tolerance, the optimal auxiliary HJB-PDE accuracy values for SLIS and MLIS may differ because the two methods balance auxiliary HJB-PDE and sampling costs differently. In SLIS, reducing the auxiliary HJB-PDE error directly improves the variance of the estimator. In contrast, in MLIS, the variance decays across levels, and the multilevel hierarchy introduces additional trade-offs. Algorithm~\ref{algMLIS} is proposed to determine the optimal MLIS work for a given relative tolerance $\mathrm{TOL} $.

As a numerical example to evaluate the effectiveness of MLIS in terms of the total computational cost compared to SLIS, this work considers the smooth setting for the observable and drift functions, employing the common likelihood formulation. This approach corresponds to the setting described in Section~\ref{common likelihood}, where MLIS can outperform SLIS across auxiliary HJB-PDE resolutions \( P \).

Algorithm~\ref{algMLIS} was applied to compute the optimal auxiliary HJB-PDE accuracy and the corresponding total work of MLIS, compared with the optimized total work of SLIS, computed using Algorithm~\ref{algSLIS}. Figure~\ref{SLvsML} displays the results, where the black lines represent the total work, dashed lines indicate SLIS, and solid lines represent MLIS. Even accounting for the cost of solving the auxiliary HJB-PDE, MLIS still demonstrates lower overall computational work than SLIS in this example.
\begin{figure}[ht]
\begin{center}
\includegraphics[scale = 0.4]{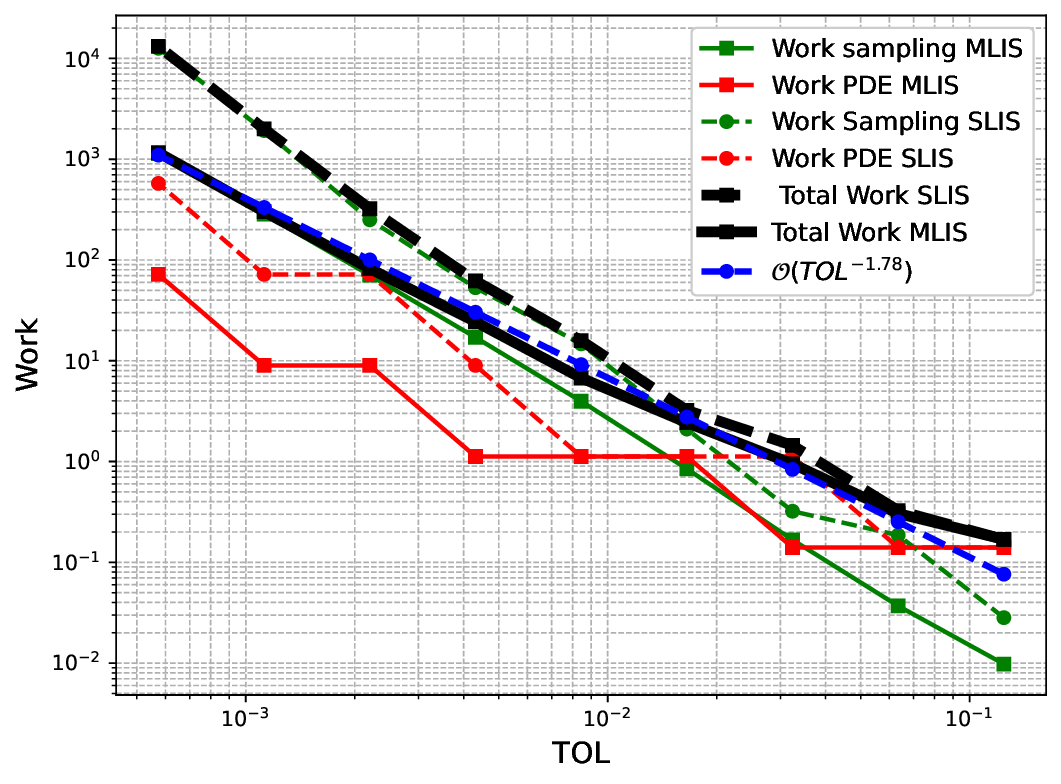}
\hspace{-0.1cm}
\caption{Optimized total work of single-level importance sampling (SLIS) vs. multilevel importance sampling (MLIS).}
\label{SLvsML}
\vspace{-3mm}
\end{center}
\end{figure}

In Figure~\ref{SLvsML}, the sampling and total work of MLIS exhibits a convergence rate of approximately \( 1.78 \), notably lower than the typical minimal rate of \( 2 \) observed in the standard MLMC method, where the single-level variance remains constant across levels. This observation indicates the need for a deeper investigation into the behavior of MLIS when the single-level variance decays with the level, as explored in the following section.
\begin{algorithm}[ht]
\caption{Optimization of the total work for multilevel importance sampling}
\label{algMLIS}
\begin{algorithmic}[1]
\State \textbf{Input:}  Target tolerance \( \mathrm{TOL} \), initial coarse level \( N_0 \), initial auxiliary HJB-PDE accuracy \( \epsilon_{\mathrm{PDE}}^0 \)
\State Compute \( N_{\mathrm{opt}} \) using~\eqref{eq:Nopt}
    \State Compute $L_{\mathrm{opt}}=\log_2 \left( \frac{N_{\mathrm{opt}}}{N_0} \right)$
\State Set \( \text{Work}_{\mathrm{MLIS}}^{\mathrm{opt}} \gets \infty \) and \( \epsilon_{\mathrm{PDE}}\gets \epsilon_{\mathrm{PDE}}^0 \)
\While{true}
  \State Compute the variance \( V_{L_{\mathrm{opt}}}^{\mathrm{IS}}(\epsilon_{\mathrm{PDE}}) \)
  \State Compute $\text{Work}_\text{sampling}^{\text{IS}}$ using \eqref{samplingwork}
    \State Compute $\ell_0^{\mathrm{opt}} (\epsilon_{\mathrm{PDE}} )$ using Algorithm~\ref{alg:N0opt}
    \If{$\ell_0^{\mathrm{opt}}=L_{\mathrm{opt}}$}
        \State Set $\text{Work}_\text{sampling}^{\mathrm{MLIS}}\gets \text{Work}_\text{sampling}^{\mathrm{IS}}$
    \Else
        \State Compute $V_{ {\ell_0^\mathrm{opt}} }^{\mathrm{IS}}(\epsilon_{\mathrm{PDE}})$
             \State Compute $\text{Work}_\text{sampling}^{\mathrm{MLIS}}$ using \eqref{MLISwork}
    \EndIf
    \State Compute  $\text{Work}_{\mathrm{PDE}}(\epsilon_{\mathrm{PDE}})$
    \State Compute total work $\text{Work}_{\mathrm{MLIS}}(\epsilon_{\mathrm{PDE}})$ using \eqref{totalWorkMLIS}
    \If{$\mathrm{Work}_{\mathrm{MLIS}} < \text{Work}_{\mathrm{MLIS}}^{\mathrm{opt}}$}
  \State \( \text{Work}_{\mathrm{MLIS}}^{\mathrm{opt}} \gets \text{Work}_{\mathrm{MLIS}} \)
        \State \( \epsilon_{\mathrm{PDE}}^{\mathrm{opt}} \gets \epsilon_{\mathrm{PDE}} \)
        \State Refine auxiliary HJB-PDE by reducing \( \epsilon_{\mathrm{PDE}} \)
    \Else
        \State \textbf{break}
    \EndIf
\EndWhile
\State \textbf{Output:} \( \epsilon_{\mathrm{PDE}}^{\mathrm{opt}}, \text{Work}_{\mathrm{MLIS}}^{\mathrm{opt}} \)
\end{algorithmic}
\end{algorithm}
\subsection{MLIS Work Rate in Decaying Single-Level Variance}
Motivated by the investigation of MLIS in settings where the single-level variance decays, this work recalls the classical complexity theorem of MLMC estimators.
\begin{theorem}[Work rate of the MLMC estimator under constant single-level variance] 
\label{thm1}
Let \( \mathrm{TOL} \) denote a prescribed target tolerance. 
Suppose that the single-level variance remains constant across levels and that positive constants \( \alpha, \beta, \gamma \) exist such that
\begin{itemize}
  \item \( |\mathbb{E}[g_w^{\ell} - q_w]| = \mathcal{O}(2^{-\alpha \ell}) \),
  \item \( V_{\ell,\ell-1} = \mathcal{O}(2^{-\beta \ell}) \), and
  \item \( C_{\ell,\ell-1} = \mathcal{O}(2^{\gamma \ell}) \).
\end{itemize}
Then, the sampling work of the MLMC method satisfies the following \cite{giles2015multilevel}:
\[\text{Work}_\text{sampling}^{\text{MLMC}}=
\begin{cases}
\mathcal{O}(\mathrm{TOL}^{-2}), & \beta > \gamma, \\
\mathcal{O}(\mathrm{TOL}^{-2} \log^2(\mathrm{TOL})), & \beta = \gamma, \\
\mathcal{O}(\mathrm{TOL}^{-2 - (\gamma - \beta)/\alpha}), & \beta < \gamma.
\end{cases}\]
\end{theorem}

In the application in Section~\ref{NumEx}, we have \( \alpha = 1 \), \( \beta = 2 \), and \( \gamma = 1 \). 
Under the assumptions of Theorem~\ref{thm1}, this finding implies an optimal computational complexity of order \( \mathcal{O}(\mathrm{TOL}^{-2}) \). However, as observed in Figure~\ref{SLvsML}, MLIS achieves a lower effective work rate of about \( \mathcal{O}(\mathrm{TOL}^{-1.78}) \), which raises the question: can the MLIS method outperform the classical MLMC complexity rate when the single-level variance decays with the level? Proposition~\ref{prop:MLISrate} explores this question by studying MLIS under conditions where the single-level variance decays with an increasing \( \ell \), and this work analyzes the resulting influence on the sampling work rate.
\begin{proposition}[Work rate of MLIS in decaying single-level variance]
\label{prop:MLISrate}
Let $\mathrm{TOL}$ denote the target accuracy. This work makes the following assumptions:
\begin{itemize}
  \item \( |\mathbb{E}[g_w^{\ell} - q_w]| = \mathcal{O}(2^{-\alpha \ell}) \),
  \item \( V^{\mathrm{IS}}_{\ell,\ell-1} = \mathcal{O}(2^{-\beta \ell}) \),
  \item \( C^{\mathrm{IS}}_{\ell,\ell-1} = \mathcal{O}(2^{\gamma \ell}) \),
    \item The single-level variance and cost at the coarse level $\ell_0(\mathrm{TOL})$ scale as follows:
    \[V^{\mathrm{IS}}_{{\ell_0}}= \mathcal{O}(\mathrm{TOL}^{v_0}), \qquad C^{\mathrm{IS}}_{{\ell_0}}= \mathcal{O}(\mathrm{TOL}^{-c_0}),\]
    where $0 \leq c_0 \leq \frac{1}{\alpha}$ and $v_0 \geq 0$.
\end{itemize}
Thus, the sampling work of MLIS satisfies the following:
\begin{itemize}
  \item If $\gamma < \beta$, 
  \[\text{Work}^{\mathrm{MLIS}}_{\mathrm{sampling}}
  = \mathcal{O}\bigl(\mathrm{TOL}^{-2+H}\bigr),
  \qquad
  H = \min\bigl(v_0 - c_0, \, (\beta-\gamma)c_0 \bigr).\]
  \item If $\gamma = \beta$,
  \[\text{Work}^{\mathrm{MLIS}}_{\mathrm{sampling}} =
  \begin{cases}
    \mathcal{O}\bigl(\mathrm{TOL}^{-2} \log^2 (\mathrm{TOL}) \bigr), & v_0-c_0 \geq 0, \\[6pt]
    \mathcal{O}\bigl(\mathrm{TOL}^{-2 + v_0-c_0}\bigr), & v_0-c_0 < 0.
  \end{cases}\]
  \item If $\gamma > \beta$, 
  \[\text{Work}^{\mathrm{MLIS}}_{\mathrm{sampling}}
  = \mathcal{O}\bigl(\mathrm{TOL}^{-2-H}\bigr),
  \qquad
  H = \max\bigl(c_0-v_0, \, \tfrac{\gamma-\beta}{\alpha} \bigr).\]
\end{itemize}
\end{proposition}
\begin{proof}
See Appendix~\ref{proof2}.
\end{proof}
\begin{remark}[MLIS vs. SLIS]
If the optimal single-level variance at the final optimal level $L_{\mathrm{opt}}$ scales as
\[
V_{L_{\mathrm{opt}}}^{\mathrm{IS}}=\mathcal{O}(\mathrm{TOL}^{v_L}),\qquad v_L\ge 0,
\]
from \eqref{samplingwork}, the SLIS sampling work is given by
\begin{equation}\label{SLISrate}
\mathrm{Work}^{\mathrm{SLIS}}_{\mathrm{sampling}}
= \mathcal{O}\!\left(\mathrm{TOL}^{-2-\frac{1}{\alpha}+v_L}\right).
\end{equation}
This rate also follows from Proposition~\ref{prop:MLISrate} by taking 
\(\ell_0=L_{\mathrm{opt}}\), yielding \(v_0=v_L\) and \(c_0=\tfrac{1}{\alpha}\).
For the MLIS rate to be better than that of SLIS, the following {necessary} conditions must hold:
\begin{itemize}
\item If \(\gamma<\beta\),
\begin{equation}
 \min\bigl(v_0 - c_0, \, (\beta-\gamma)c_0 \bigr) \;\geq\; v_L-\tfrac{1}{\alpha}
\qquad\Longrightarrow\qquad
v_0-c_0 \;\geq\; v_L-\tfrac{1}{\alpha}
\ \text{and}\ 
(\beta-\gamma)c_0 \;\geq\; v_L-\tfrac{1}{\alpha}.
\end{equation}
\item If \(\gamma\ge\beta\),
\begin{equation}
v_L \;\leq\; \frac{1-(\gamma-\beta)}{\alpha}.
\end{equation}
\end{itemize}
Under an optimized MLIS strategy, these conditions are automatically satisfied because an optimal MLIS cannot perform worse than SLIS.
\end{remark}
\begin{remark}[MLIS vs. MLMC]
The potential for MLIS to improve on the standard MLMC complexity rate exists only when 
\( \gamma < \beta \) and 
\(
H = \min\bigl(v_0 - c_0, \, (\beta-\gamma)c_0 \bigr) > 0.
\)
This situation requires \( v_0 - c_0 > 0 \), meaning that \( V_{\ell_0}^{\mathrm{IS}} C_{\ell_0}^{\mathrm{IS}} \) decays with~$\mathrm{TOL}$. The optimal choice of \(\ell_{0}\) should ensure \(v_0 - c_0 \geq 0\). If \(v_0 - c_0 = 0\), then MLIS reduces to the standard MLMC method. As observed in the previous numerical example \ref{NumEx}, the regime \(\gamma < \beta\) explains the improved convergence rate relative to the classical MLMC method. 

When \(\gamma \geq \beta\), no improvement over the standard MLMC rate is possible: the complexity coincides with the classical MLMC results of~\cite{giles2015multilevel}, regardless of the variance decay. The same conclusion applies: by optimality, \(v_0 - c_0 \geq 0\), and if equality holds, MLIS reduces to the MLMC method. 
\end{remark}
\begin{corollary}[Work rate of MLIS with optimal choice of the coarse level]
Let $\mathrm{TOL}$ denote the target accuracy. This work assumes the following:
\begin{itemize}
  \item \( |\mathbb{E}[g_w^{\ell} - q_w]| = \mathcal{O}(2^{-\alpha \ell}) \),
  \item \( V^{\mathrm{IS}}_{\ell,\ell-1} = \mathcal{O}(2^{-\beta \ell}) \),
  \item \( C^{\mathrm{IS}}_{\ell,\ell-1} = \mathcal{O}(2^{\gamma \ell}) \),
  
    \item $\ell_0^{\mathrm{opt}}(\mathrm{TOL})$ denotes the optimal coarse level, and \[V^{\mathrm{IS}}_{{\ell_0^{\mathrm{opt}}}}= \mathcal{O}(\mathrm{TOL}^{v_0^{\mathrm{opt}}}), \qquad C^{\mathrm{IS}}_{{\ell_0^{\mathrm{opt}}}}= \mathcal{O}(\mathrm{TOL}^{-c_0^{\mathrm{opt}}}),
    \]
    where $0 \leq c_0^{\mathrm{opt}} \leq \frac{1}{\alpha}$ and $v_0^{\mathrm{opt}} \geq 0$.
\end{itemize}
Thus, $ v_0^{\mathrm{opt}} - c_0^{\mathrm{opt}} \geq 0$, and the optimal sampling work of MLIS satisfies the following:
\[ \text{Work}_\text{sampling}^{\text{MLIS}}=
\begin{cases}
\mathcal{O}\bigl(\mathrm{TOL}^{-2+H}\bigr),
  \qquad
  H = \min\bigl(v_0^{\mathrm{opt}}  -  c_0^{\mathrm{opt}} , \, (\beta-\gamma) c_0^{\mathrm{opt}}  \bigr), & \beta > \gamma, \\
\mathcal{O}(\mathrm{TOL}^{-2} \log^2(\mathrm{TOL})), & \beta = \gamma, \\
\mathcal{O}(\mathrm{TOL}^{-2 - (\gamma - \beta)/\alpha}), & \beta < \gamma.
\end{cases}
\]
\end{corollary}
This work estimates the rates \( v_0^{\mathrm{opt}}  \) and \( c_0^{\mathrm{opt}} \) in the fourth assumption numerically to validate the theoretical rate in Proposition~\ref{prop:MLISrate}. At the optimal coarse level \( \ell_0^{\mathrm{opt}} \), Figure~\ref{ratesvc} depicts the dependence on \(\mathrm{TOL}\) of the single-level variance \( V_{{\ell_0^{\mathrm{opt}}}}^{\mathrm{IS}} \), the single-level cost 
\( C_{{\ell_0^{\mathrm{opt}}}}^{\mathrm{IS}} \), and their product 
\( V_{{\ell_0^{\mathrm{opt}}}}^{\mathrm{IS}} C_{{\ell_0^{\mathrm{opt}}}}^{\mathrm{IS}} \).
\begin{figure}[ht]
\centering
\begin{minipage}[b]{0.32\linewidth}
    \centering
    \includegraphics[width=\linewidth]{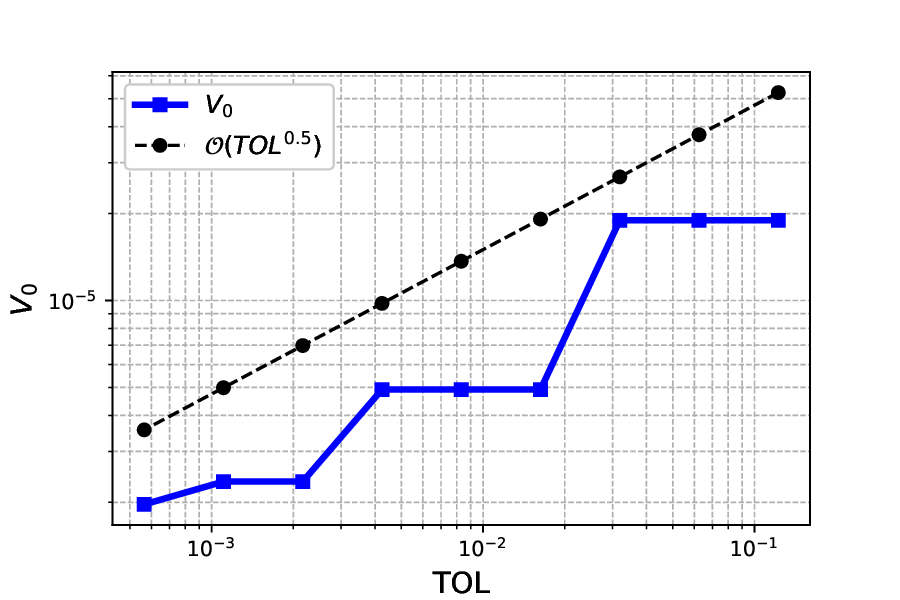}\\
    Variance=\( V_{{\ell_0^{\mathrm{opt}}}}^{\mathrm{IS}} \)
\end{minipage}\hfill
\begin{minipage}[b]{0.32\linewidth}
    \centering
    \includegraphics[width=\linewidth]{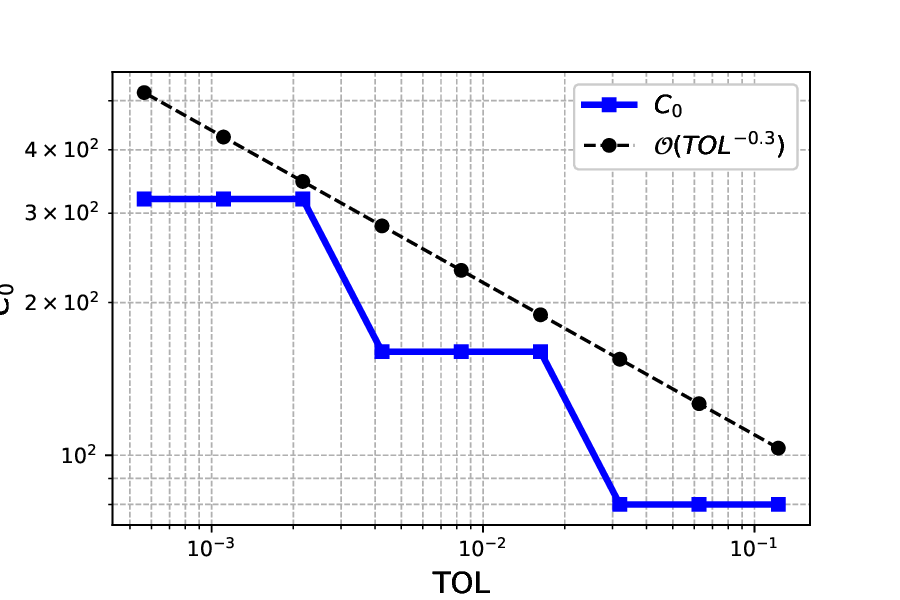}\\
    Cost=\( C_{{\ell_0^{\mathrm{opt}}}}^{\mathrm{IS}} \)
\end{minipage}\hfill
\begin{minipage}[b]{0.32\linewidth}
    \centering
    \includegraphics[width=\linewidth]{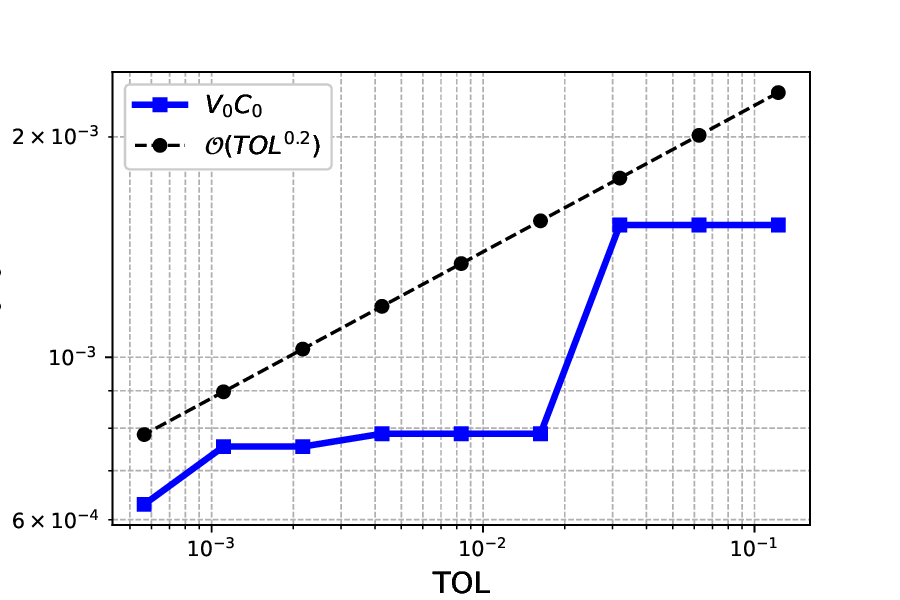}\\
    Product=\( V_{{\ell_0^{\mathrm{opt}}}}^{\mathrm{IS}} C_{{\ell_0^{\mathrm{opt}}}}^{\mathrm{IS}} \)
\end{minipage}
\caption{Dependence of the variance, cost, and their product on tolerance at the optimal coarse level.}
\label{ratesvc}
\end{figure}
These plots yield the following: \(
v_0^{\mathrm{opt}}  = 0.5, \quad c_0^{\mathrm{opt}}  = 0.3, \quad v_0^{\mathrm{opt}}  - c_0^{\mathrm{opt}}  = 0.2.
\)

Applying Proposition~\ref{prop:MLISrate} with 
\( \beta = 2 \) and \( \gamma = 1 \) results in \[
H = \min \bigl(v_0^{\mathrm{opt}} - c_0^{\mathrm{opt}} ,\; (\beta - \gamma)c_0\bigr)
= \min (0.2, 0.3) = 0.2,
\]
so that
\[
\text{Work}^{\mathrm{MLIS}}_{\mathrm{sampling}}
= \mathcal{O}(\mathrm{TOL}^{-2+0.2})
= \mathcal{O}(\mathrm{TOL}^{-1.8}).
\]
This theoretical rate aligns with the convergence rate observed in the numerical experiments in Figure \ref{SLvsML}, confirming the analytical validity of this method.
\section{Error Analysis in the presence of Smoothing}
\label{error analysis}
Recall that in Section~\ref{smooth section}, we introduced the smooth approximations \(f^{d}\) and \(g_{w}^{c}\) (see~\eqref{fdgc}) to improve the
multilevel variance convergence rates. In the numerical experiments presented
earlier, the smoothing parameters were kept fixed for each value of
\(\mathrm{TOL}\). In this section, we first show that the smoothing error
\(\Delta^{c,d}\) admits a quadratic dependence on the parameters \(c\) and \(d\),
which leads to an optimal choice of these parameters proportional to \(\sqrt{\mathrm{TOL}}\). We then demonstrate that, in particular cases, such as the
occupation time example, an alternative strategy can be used: although \(c\) and
\(d\) may be chosen independently of \(\mathrm{TOL}\), selecting them in a
specific dependent relation can yield an almost vanishing smoothing error.
\subsection{Characterization of the Smoothing Error}
The smoothing introduces the approximation error:
\begin{equation}\label{SmoothingError}
\Delta^{c,d}=q_w-q_w^{c,d}.
\end{equation}
By inserting and subtracting the intermediate term $\mathbb{E}[g^{c}_{w}(Z(T))]$, we obtain the decomposition
\begin{equation}\label{errdecomp}
\Delta^{c,d} 
= \Delta^{c,0}+\Delta^{d}_{c},
\end{equation}
where 
\begin{equation}
\Delta^{c,0}=\mathbb{E}[g_w(Z(T))]-\mathbb{E}[g^{c}_{w}(Z(T))],
\end{equation}
\begin{equation}
\Delta^{d}_{c}=\mathbb{E}[g^{c}_{w}(Z(T))]-\mathbb{E}[g^{c}_{w}(Z^{d}(T))].
\end{equation}
Proposition~\ref{smg} establishes that the smoothing error associated with \(g_{w}^{c}\), denoted \(\Delta^{c,0}\), scales quadratically in \(c\). An analogous quadratic behavior for the smoothing error $\Delta^{d}_{c}$ associated with \(f^d\), when tested using the smoothed observable \(g_{w}^{c}\), is established in Proposition~\ref{smf}.  
\begin{proposition}[Smoothing error for \(g_{w}^{c}\)]
\label{smg}
Let \(\rho_{T}\) be the density of \(Z(T)\).
Assume that \(\rho_{T}\) is continuously differentiable in a neighborhood of \(w\). 
Then, as \(c \downarrow 0\),
\[
\Delta^{c,0}=\mathbb{E}\!\left[g_{w}(Z(T))\right]-\mathbb{E}\!\left[g_{w}^{c}(Z(T))\right]
 = \frac{c^{2}}{6}\,\rho_{T}'(w) + o(c^{2}).
\]
\end{proposition}
\begin{proof}
See Appendix~\ref{proof3}.
\end{proof}
\begin{proposition}[Smoothing error for \(f^{d}\)]
\label{smf}
Recall that \((\boldsymbol{X}(s),Z(s))\) solves \eqref{SDEgdplus}. Let \(\rho_{s}\) be the density of \(Z(s)\). Let \(u^{c}(t,\boldsymbol{x},z)\) denote the solution of the KBE associated 
with the smoothed payoff \(g_{w}^{c}\) \eqref{fdgc}. Assume that Assumptions~\ref{ass1}, \ref{ass2}  and \ref{ass3}  stated in Appendix \ref{proof4} hold.
Then, as \(d\downarrow0\),
\[
\Delta_{c}^{d}= O(d^{2}).
\]
\end{proposition}
\begin{proof}
See Appendix~\ref{proof4}.
\end{proof}
\begin{remark}[Regularity Requirements]
\label{remreg}
Assumptions~\ref{ass1}, \ref{ass2} and \ref{ass3} impose regularity and integrability conditions on 
\(H_s^c(z)=\mathbb{E}\!\left[\partial_{z}u^{c}(s,\boldsymbol{X}(s),Z(s))\mid Z(s)=z\right]\rho_s(z)\), and therefore ultimately on the density \( \rho_s \) 
and the value function \(u^c\).  
In the occupation-time setting considered in this work, 
the discontinuous structure of the drift makes a rigorous verification of these 
assumptions delicate and beyond the scope of this paper. 
For this reason, the quadratic behavior predicted by Proposition~\ref{smf} 
is assessed empirically through numerical experiments in 
Section~\ref{NumQuadr}.
\end{remark}

Using the decomposition \eqref{errdecomp} together with Propositions~\ref{smg} and~\ref{smf}, and assuming that the respective conditions of these propositions are satisfied, we obtain that the smoothing error satisfies
\begin{equation}
\label{smerrfinal}
\Delta^{c,d}
= C_{1}c^{2} + C_{2}d^{2} + o(c^{2}) + o(d^{2}),
\end{equation}
where \(C_{1},C_{2}\in\mathbb{R}\) and \(C_{2}\) may depend on \(c\). These constants can be estimated numerically and used to choose \(c\) and \(d\) to meet a prescribed tolerance. 
\subsection{Total Error Splitting}
\label{errsplit}
Let $\mathcal{A}_{\mathrm{IS}}^{c,d}$ denote the IS estimator defined similarly to \eqref{ISestimator}, but applied to the smoothed functions $g_w^c$ and $f^d$. The total absolute error $\epsilon_T^{c,d}$ satisfies
\begin{equation}
\epsilon_T^{c,d}
= \bigl| q_w - \mathcal{A}_{\mathrm{IS}}^{c,d} \bigr|
\leq  \bigl|\Delta^{c,d}\bigr| + \epsilon_b^{c,d} + \epsilon_s^{c,d},
\end{equation}
where $\epsilon_b^{c,d}$ and $\epsilon_s^{c,d}$ denote, respectively, the bias and the statistical error associated with estimating $q_w^{c,d}$. To achieve an overall relative accuracy of order $\mathrm{TOL}$, we impose:
\begin{equation}
\epsilon_b^{c,d} \leq \theta_{1}\, q_w \mathrm{TOL}, \qquad
\epsilon_s^{c,d}\leq \theta_{2}\,q_w \mathrm{TOL}, \qquad
|\Delta^{c,d}| \leq  \theta_{3}\, q_w \mathrm{TOL}, \qquad
\theta_{1}+\theta_{2}+\theta_{3}=1.
\end{equation}
This decomposition ensures that none of the error sources dominates the others, allowing for optimal parameter selection across numerical, stochastic, and smoothing components.

The condition $|\Delta^{c,d}| \leq  \theta_3\,q_w \mathrm{TOL}$ can be decomposed using
\(|\Delta^{c,d}| \le |\Delta^{c,0}| + |\Delta^{d}_{c}|.\)
Assigning a fraction $\tilde{\theta}\in(0,1)$ of the smoothing tolerance to $|\Delta^{c,0}|$ and the remaining fraction $1-\tilde{\theta}$ to $|\Delta^{d}_{c}|$, we obtain
\begin{equation}\label{choicec}
|\Delta^{c,0}| = \tilde{\theta}\,\theta_{3}\,q_w \mathrm{TOL}
\quad\Longrightarrow\quad
c = \sqrt{\frac{\tilde{\theta}\,\theta_{3}\,q_w \mathrm{TOL}}{|C_{1}|}},
\end{equation}
\begin{equation}\label{choiced}
|\Delta^{d}_{c}| = (1-\tilde{\theta})\,\theta_{3}\,q_w \mathrm{TOL}
\quad\Longrightarrow\quad
d = \sqrt{\frac{(1-\tilde{\theta})\,\theta_{3}\,q_w \mathrm{TOL}}{|C_{2}|}}.
\end{equation}
Finally, choosing $c$ and $d$ according to \eqref{choicec}--\eqref{choiced} guarantees that
\begin{equation}
\left| \Delta^{c,d} \right|<\theta_3 q_w \mathrm{TOL}.
\end{equation}
An optimal selection of
$\theta_{1}$, $\theta_{2}$, $\theta_{3}$, and $\tilde{\theta}$ may in principle
be obtained by minimizing the total computational cost with respect to these
variables. 
\subsection{Numerical Verification of the Quadratic Smoothing Error}
\label{NumQuadr}
Returning to the fade duration example of Section~\ref{section num}, a rigorous
justification of the quadratic form \eqref{smerrfinal} would require verifying
the assumptions underlying Propositions~\ref{smg} and~\ref{smf}. As noted in
Remark~\ref{remreg}, establishing these properties is
nontrivial in the occupation-time setting. For this reason, we focus here on numerically validating the predicted
quadratic behavior of the smoothing error.

We begin with the quadratic form of $\Delta^{c,0}$ and the estimation of $C_{1}$.
Reference values of
$q_{w}^{c,0}$ are computed for several choices of $c$, together with the
corresponding errors $\Delta^{c,0}$. The simulations use a very fine time discretization and a large number of IS samples, ensuring that the combined absolute discretization and statistical errors are negligible (approximately $10^{-6}$ in total).
Consequently, the observed deviation is dominated by the
smoothing error. 

The calculated values of $\Delta^{c,0}$ are observed to be negative.
Figure~\ref{errc} reports the smoothing error $-\Delta^{c,0}$ on a log--log scale.
The quadratic scaling of $\Delta^{c,0}$ is confirmed, and the constant $C_{1}$ is
obtained from the slope. From Figure~\ref{errc}, we obtain the estimate
\begin{equation}
\label{C1}
\hat{C}_{1} = -4.29 \times 10^{-3}.
\end{equation}

\begin{figure}[ht]
\begin{center}
\includegraphics[scale = 0.4]{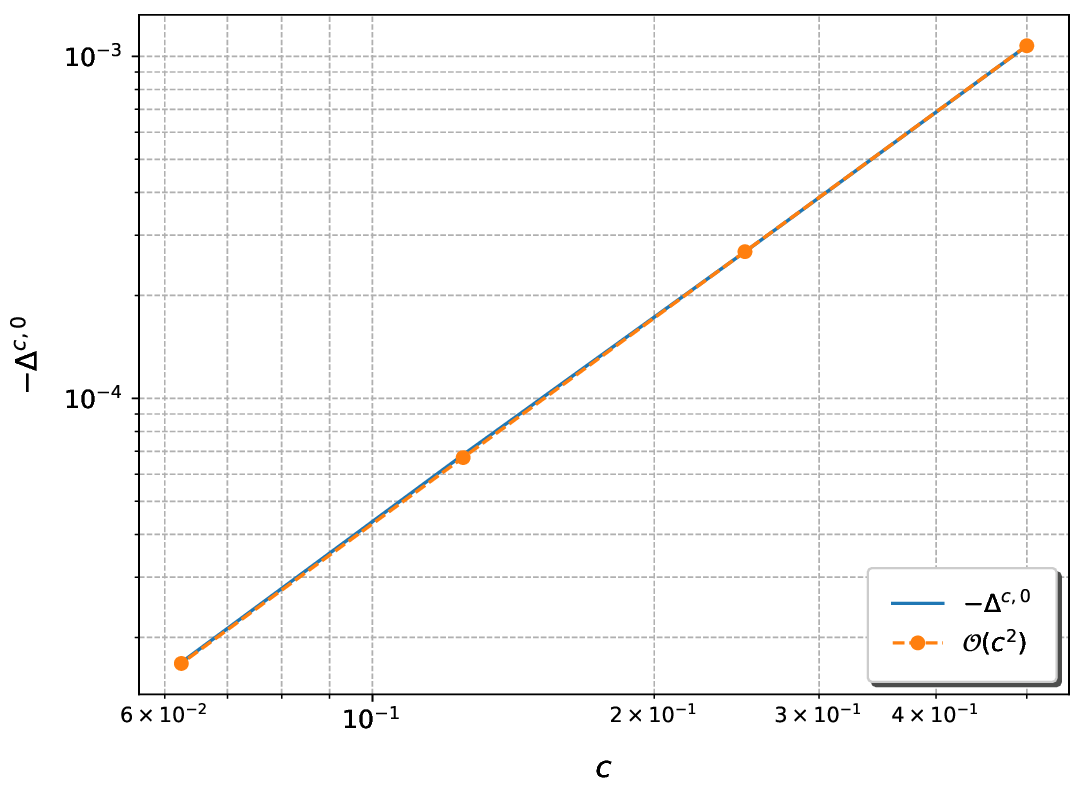}
\hspace{-0.1cm}
\caption{Verification of the quadratic behavior of the smoothing error of $g_w^c$.}
\label{errc}
\vspace{-3mm}
\end{center}
\end{figure}
A similar procedure is used to confirm the quadratic form of $\Delta_c^d$ and estimate $C_{2}$. To examine whether $C_{2}$
depends on the value of $c$, we compute reference values of $q_{w}^{c,d}$ and
the corresponding errors $\Delta_{c}^{d}$ for various pairs $(c,d)$. The results
are shown in Figure~\ref{errd}. From the figure, we observe that $C_{2}$ does not
depend on the smoothing applied to $g_{w}$. Moreover, the quadratic scaling of
$\Delta_{c}^{d}$ is confirmed, and the resulting ratio yields the estimate
\begin{equation}
\label{C2}
\hat{C}_{2} = 7.65\times 10^{-2}.
\end{equation}
\begin{figure}[ht]
\begin{center}
\includegraphics[scale = 0.4]{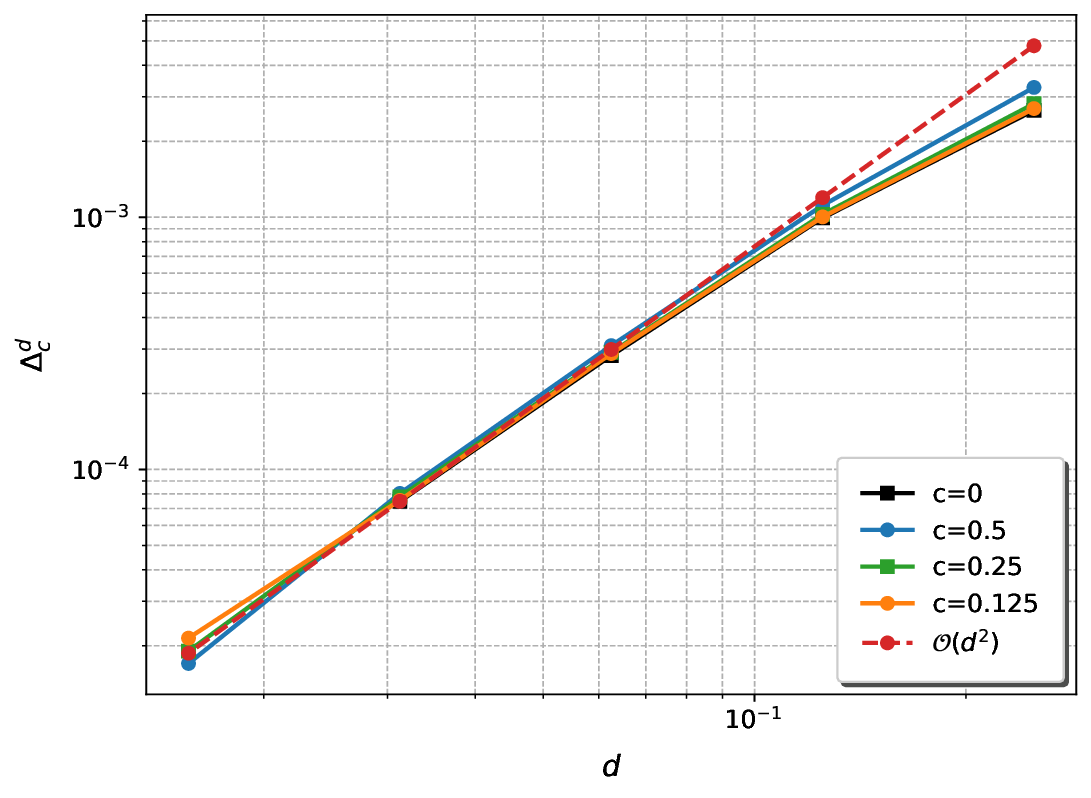}
\hspace{-0.1cm}
\caption{Verification of the quadratic behavior of the smoothing error of $f^d$.}
\label{errd}
\vspace{-3mm}
\end{center}
\end{figure}
\subsection{Choice of the smoothing parameters}
\label{choicepar}
With numerical values of $C_{1}$ and $C_{2}$ available, we now address the
choice of the smoothing parameters $c$ and $d$ required to approximate $q_{w}$
with relative accuracy $\mathrm{TOL}$. 
Although one might naturally choose the smoothing parameters $c$ and $d$
directly from~\eqref{choicec}--\eqref{choiced}, an interesting feature emerges
in the fade duration example: the constants $\hat{C_{1}}$ in \eqref{C1} and $\hat{C_{2}}$ in \eqref{C2} have opposite
sign. This implies that the leading order terms in the smoothing errors associated with $g_w$ and $f$ may cancel. In fact, one can enforce
\(
\Delta^{c,0} = - \Delta_{c}^{d},
\)
which yields the relation
\begin{equation}
\label{choicecd}
d = c \sqrt{-\frac{C_{1}}{C_{2}}}.
\end{equation}
This observation shows that what matters is not how small $c$ and $d$ are
individually, but how they are related. However, Propositions~\ref{smg} and
\ref{smf} still require both parameters to be sufficiently small for the
quadratic regime to hold. 

In practice, the leading order terms of the smoothing error will not vanish exactly because we only use
estimates $\hat{C}_{1}$ and $\hat{C}_{2}$ rather than the exact constants. Thus
the remaining error arises not from the smoothing itself, but from the
imperfect estimation of the constants. The resulting absolute smoothing error is
\begin{equation}
\begin{aligned}
|\Delta^{c,d}|
  &= \left| -\left( \frac{\hat{C}_{2}}{\hat{C}_{1}} C_{1}\, d^{2} \right) + C_{2} d^{2} \right|+o(d^{2}) \\
  &= d^{2} \left|\, C_{2} - \frac{\hat{C}_{2}}{\hat{C}_{1}}\, C_{1} \right|+o(d^{2}).
\end{aligned}
\end{equation}
In the numerical example, we take $c=0.5$, which yields $d=0.122$ through
\eqref{choicecd}. The smoothing error $\Delta^{c,d}$ is then computed
numerically for this pair $(c,d)$ using a reference value obtained at
$c=0$, $d=0$, with a very fine time discretization and an extremely large
number of samples. For this choice, we obtain an absolute smoothing error
\(
|\Delta^{c,d} |= 1.12\times 10^{-7},
\)
and the corresponding relative error is
\(
\frac{\Delta^{c,d}}{q_w} = 3.9 \times 10^{-5}
\).

Finally, the use of fixed smoothing parameters ($d = 0.122$ and $c = 0.5$) over
the entire range of $\mathrm{TOL}$ considered in Section~\ref{smooth section} and
Section~\ref{NumEx} is now justified. In particular, the SLIS and MLIS methods
with smoothing, shown in Figure~\ref{SLvsML}, do not introduce any appreciable
bias compared with the non-smoothed case, which corresponds to the original
problem, namely the fade duration \eqref{equ}.
\section*{Conclusion}
This work addressed the rare-event estimation problem for the CCDF of the
occupation time of systems modeled by SDEs. By leveraging the connection between IS and SOC, we developed an optimal SLIS estimator
and incorporated the preprocessing cost of the auxiliary HJB equation into a
unified analysis of computational efficiency. This allowed us to quantify the
trade-off between control accuracy and variance reduction and to determine the
optimal balance that minimizes the total computational work.

Building on the SLIS methodology, we extended the approach to a MLIS framework. A key observation in this setting is that the single-level
variance decays with the refinement of the discretization, a behavior rooted in
the zero-variance property of the optimal control. This motivated the
derivation of a necessary and sufficient condition under which MLIS
outperforms SLIS. To satisfy this condition, we introduced a smoothing of both
the drift and the observable, together with a common-likelihood MLIS
formulation designed to preserve variance-decay under IS. We
also established that the smoothing error can be rigorously controlled and is
negligible, ensuring that the resulting estimator remains unbiased for the
original occupation-time problem.

Our theoretical analysis further showed that the classical MLMC complexity
framework extends naturally to settings where the coarse-level variance depends on the target accuracy, and we identified scenarios in which MLIS achieves
computational rates strictly better than those available in standard MLMC.
Numerical experiments on fade duration estimation confirmed these theoretical
predictions, demonstrating substantial gains in efficiency and validating the
practical relevance of the proposed methodology.

Overall, this work shows that combining SOC-based IS, smoothing,
and multilevel methods, while explicitly accounting for preprocessing
costs, provides a powerful strategy for rare-event estimation in SDEs.
\textbf{Acknowledgments} This publication is based on work supported by the King Abdullah University of Science and Technology (KAUST) Office of Sponsored Research (OSR) under Award No. OSR-2019-CRG8-4033 and the Alexander von Humboldt Foundation. 
\appendix
\section{Proof of Proposition \ref{prop1}}
\label{proof1}
First, ~\eqref{condMLnewcase1} is equivalent to the following inequality:
\begin{equation}
\label{condML}
\sqrt{V_{\ell+1} \, C_{{\ell+1}}} > \sqrt{V_\ell \, C_{{\ell}}} + \sqrt{V_{\ell+1,\ell} \, C_{\ell+1,\ell}}.
\end{equation}
This equivalence can be verified by making the following substitutions:
\begin{equation}
C_{\ell} = C_{0} 2^\ell, \quad C_{{\ell+1}} = C_{0} 2^{\ell+1}, \quad C_{\ell+1,\ell} = 3C_{0}2^\ell,
\end{equation}
and simplifying both sides of~\eqref{condML}.

\medskip
\noindent
\textbf{(1) Necessity:}
Let us assume that condition in~\eqref{condMLnewcase1} is not satisfied for any \( \ell_0 < L_{\mathrm{opt}} \). Then, for all  \( \ell<L_{\mathrm{opt}} \), inequality~\eqref{condML} is reversed:
\begin{equation}
\label{reversed}
\sqrt{V_{\ell+1} \, C_{{\ell+1}}} \leq \sqrt{V_{\ell} \, C_{\ell}} + \sqrt{V_{\ell+1,\ell} \, C_{\ell+1,\ell}}.
\end{equation}
Applying this recursively starting from \( \ell = L_{\mathrm{opt}} - 1 \), yields
\begin{equation}
\sqrt{V_{{L_{\mathrm{opt}}}} \, C_{{L_{\mathrm{opt}}}}} \leq \sqrt{V_{{L_{\mathrm{opt}} - 1}} \, C_{{L_{\mathrm{opt}} - 1}}} + \sqrt{V_{L_{\mathrm{opt}},L_{\mathrm{opt}} - 1} \, C_{L_{\mathrm{opt}},L_{\mathrm{opt}} - 1}},
\end{equation}
\begin{equation}
\sqrt{V_{{L_{\mathrm{opt}} - 1}} \, C_{{L_{\mathrm{opt}} - 1}}} \leq \sqrt{V_{{L_{\mathrm{opt}} - 2}} \, C_{{L_{\mathrm{opt}} - 2}}} + \sqrt{V_{L_{\mathrm{opt}} - 1, L_{\mathrm{opt}} - 2} \, C_{L_{\mathrm{opt}} - 1, L_{\mathrm{opt}} - 2}},
\end{equation}
Repeating this backward until any level \( \ell < L_{\mathrm{opt}} \) yields
\begin{equation}
\sqrt{V_{{L_{\mathrm{opt}}}} \, C_{{L_{\mathrm{opt}}}}} \leq \sqrt{V_{{\ell}} \, C_{{\ell}}} + \sum_{\ell = \ell + 1}^{L_{\mathrm{opt}}} \sqrt{V_{\ell,\ell-1} \, C_{\ell,\ell-1}}.
\end{equation}
This result implies that, starting from any coarse level \(\ell\),
\begin{equation}
\text{Work}_\text{sampling}^{\text{SLMC}} \leq \text{Work}_\text{sampling}^{\text{MLMC}},
\end{equation}
with equality occurring only when \( \ell = L_{\mathrm{opt}} \). Hence, no gain is achieved by the MLMC method unless the condition in~\eqref{condMLnewcase1} is satisfied for some \( \ell_0 < L_{\mathrm{opt}} \).
\medskip
\noindent
\\
\textbf{(2) Sufficiency:}
Next, let us assume that the condition in~\eqref{condMLnewcase1} is satisfied for some \( \ell_0 < L_{\mathrm{opt}} \) and for all \( \ell \geq \ell_0 \). Therefore, inequality~\eqref{condML} holds for all \( \ell \geq \ell_0 \). By induction, this approach obtains the following:
\begin{equation}
\sqrt{V_{{L_{\mathrm{opt}}}} \, C_{{L_{\mathrm{opt}}}}} > \sqrt{V_{{\ell_0}} \, C_{{\ell_0}}} + \sum_{\ell = \ell_0 + 1}^{L_{\mathrm{opt}}} \sqrt{V_{\ell,\ell-1} \, C_{\ell,\ell-1}}.
\end{equation}
This result implies that, starting with a coarse level $\ell_0$,
\begin{equation}
\text{Work}_\text{sampling}^{\text{MLMC}} < \text{Work}_\text{sampling}^{\text{SLMC}},
\end{equation}
that is, the MLMC method achieves less sampling work than the SLMC method when initialized at level \( \ell_0 \).
\section{Proof of Corollary \ref{cor}}
\label{proofcor}
From Proposition~\ref{prop1}, if the MLMC method is strictly better than the SLMC method, then a level \( \ell_0 < L_{\mathrm{opt}} \) exists such that
\[\sqrt{2V_{\ell_0+1}} > \sqrt{V_{\ell_0}} + \sqrt{3\,V_{\ell_0+1,\ell_0}}.\]
Dividing this result throughout by \( \sqrt{V_0} \) yields
\[\sqrt{\frac{2 V_{\ell_0+1}}{V_0}} > \sqrt{\frac{V_{\ell_0}}{V_0}} 
+ \sqrt{\frac{3 V_{\ell_0+1,\ell_0}}{V_0}}.\]
By Assumption~(1),
\[1 - \varepsilon \leq \frac{V_0}{V_\ell} \leq 1 + \varepsilon
\quad \Rightarrow \quad \frac{1}{1 + \varepsilon} \leq \frac{V_\ell}{V_0} \leq \frac{1}{1 - \varepsilon}.\]
Employing the lower bound \( \frac{V_{\ell_0}}{V_0} \geq \frac{1}{1+\varepsilon} \) and upper bound \( \frac{V_{\ell_0+1}}{V_0} \leq \frac{1}{1-\varepsilon} \) results in the following:
\[\sqrt{\frac{2}{1 - \varepsilon}} > \sqrt{\frac{1}{1 + \varepsilon}} 
+ \sqrt{\frac{3 V_{\ell_0+1,\ell_0}}{V_0}}.\]
Rearranging this result yields \eqref{corollaryeq}, completing the proof.
\section{Proof of Proposition \ref{prop:MLISrate}}
\label{proof2}
From the MLIS sampling complexity estimate,
\[\text{Work}^{\mathrm{MLIS}}_{\text{sampling}} 
< \frac{1}{q_w^2} \left( \frac{2C}{\mathrm{TOL}} \right)^2 
\left( \sqrt{V^{\mathrm{IS}}_{{\ell_0}} C^{\mathrm{IS}}_{{\ell_0}}} 
+ \sum_{\ell = \ell_0 + 1}^{L_{\mathrm{opt}}} 2^{\frac{\gamma - \beta}{2}\ell}\right)^2.\]
We define the following:
\[
\Sigma := \sqrt{V^{\mathrm{IS}}_{{\ell_0}} C^{\mathrm{IS}}_{{\ell_0}}} 
+ \sum_{\ell = \ell_0+ 1}^{L_{\mathrm{opt}}} 2^{\frac{\gamma - \beta}{2}\ell}.
\]
Using scaling assumptions yields
\[
\sqrt{V^{\mathrm{IS}}_{{\ell_0}}
C^{\mathrm{IS}}_{{\ell_0}}} 
= \mathcal{O}\bigl(\mathrm{TOL}^{\frac{v_0 - c_0}{2}}\bigr).\] 
We analyze the sum according to the relationship between $\gamma$ and $\beta$. 

\medskip
\noindent\textbf{Case 1: $\gamma < \beta$.}
In this constraint, $2^{\frac{\gamma-\beta}{2}\ell}$ decays with~$\ell$, so the sum is 
dominated by its first term:
\[\sum_{\ell = \ell_0 + 1}^{L_{\mathrm{opt}}} 
2^{\frac{\gamma - \beta}{2}\ell}
\approx 2^{\frac{\gamma - \beta}{2}\ell_0}
= \mathcal{O}\bigl(\mathrm{TOL}^{\frac{(\beta-\gamma)c_0}{2}}\bigr).\]
Thus,
\[
\Sigma = \mathcal{O}\bigl(\mathrm{TOL}^{\frac{v_0 - c_0}{2}}\bigr)
+ \mathcal{O}\bigl(\mathrm{TOL}^{\frac{(\beta-\gamma)c_0}{2}}\bigr),
\]
and
\[
\text{Work}^{\mathrm{MLIS}}_{\text{sampling}} 
= \mathcal{O}\bigl(\mathrm{TOL}^{-2+H}\bigr), \qquad
H = \min\bigl(v_0 - c_0,\, (\beta - \gamma)c_0\bigr).
\]
\medskip
\noindent\textbf{Case 2: $\gamma = \beta$.}
In this case, 
\(
\sum_{\ell = \ell_0 + 1}^{L_{\mathrm{opt}}} 
2^{\frac{\gamma - \beta}{2}\ell} = \mathcal{O}(L_{\mathrm{opt}} - \ell_0)
=\mathcal{O}(\log \mathrm{TOL})\).
Hence,
\[
\Sigma = \mathcal{O}\bigl(\mathrm{TOL}^{\frac{v_0 - c_0}{2}}\bigr)
+ \mathcal{O}(|\log \mathrm{TOL}|).
\]
If $v_0 - c_0 \geq 0$, the logarithmic term dominates:
\[
\text{Work}^{\mathrm{MLIS}}_{\text{sampling}}
= \mathcal{O}\bigl(\mathrm{TOL}^{-2} \log^2 \mathrm{TOL}\bigr).
\]
Otherwise, if $v_0 - c_0 < 0$, the first term dominates:
\[
\text{Work}^{\mathrm{MLIS}}_{\text{sampling}}
= \mathcal{O}\bigl(\mathrm{TOL}^{-2 + (v_0 - c_0)}\bigr).
\]

\medskip
\noindent\textbf{Case 3: $\gamma > \beta$.}
In this case, $2^{\frac{\gamma-\beta}{2}\ell}$ increases with~$\ell$, so the sum is 
dominated by its last term:
\[
\sum_{\ell = \ell_0+ 1}^{L_{\mathrm{opt}}} 
2^{\frac{\gamma - \beta}{2}\ell}
\approx 2^{\frac{\gamma - \beta}{2} L_{\mathrm{opt}}}.
\]
From the bias constraint $2^{-\alpha L_{\mathrm{opt}}} \sim \mathrm{TOL}$,
\(
2^{\frac{\gamma-\beta}{2} L_{\mathrm{opt}}}
= \mathcal{O}\bigl(\mathrm{TOL}^{-\frac{\gamma - \beta}{2 \alpha}}\bigr).
\)
Therefore,
\[
\Sigma = \mathcal{O}\bigl(\mathrm{TOL}^{\frac{v_0 - c_0}{2}}\bigr)
+ \mathcal{O}\bigl(\mathrm{TOL}^{-\frac{\gamma - \beta}{2 \alpha}}\bigr),
\]
and
\[
\text{Work}^{\mathrm{MLIS}}_{\text{sampling}}
= \mathcal{O}\bigl(\mathrm{TOL}^{-2+H}\bigr),
\qquad
H = \min\bigl(v_0 - c_0,\, \tfrac{\beta - \gamma}{\alpha}\bigr).
\]
\medskip
These three cases combined complete the proof.
\section{Proof of Proposition \ref{smg}}
\label{proof3}
By definition,
\[
\mathbb{E}\!\left[g_{w}(Z(T))\right]-\mathbb{E}\!\left[g_{w}^{c}(Z(T))\right]
=\int_{\mathbb{R}}\!\big(g_{w}(z)-g_{w}^{c}(z)\big)\,\rho_{T}(z)\,dz,
\]
and the integrand vanishes outside \((w-c,w+c)\). 
With the change of variables \(x=w+u\), \(u\in(-c,c)\),
\[
\mathbb{E}\!\left[g_{w}(Z(T))\right]-\mathbb{E}\!\left[g_{w}^{c}(Z(T))\right]
=\int_{-c}^{c}\!\Big(\mathbbm{1}_{\{u \ge 0\}}-\tfrac{1}{2}-\tfrac{u}{2c}\Big)\,\rho_{T}(w+u)\,du.
\]
Splitting at \(u=0\) and using the expansion \(\rho_{T}(w+u)=\rho_{T}(w)+u\rho_{T}'(w)+o(u)\) yields
\[
\begin{aligned}
\mathbb{E}\!\left[g_w(Z(T))\right]-\mathbb{E}\!\left[g_w^c(Z(T))\right]
&=\int_{0}^{c}\!\Big(\tfrac{1}{2}-\tfrac{u}{2c}\Big)\big(\rho_T(w)+u\rho_T'(w)\big)\,du
\\ & +\int_{-c}^{0}\!\Big(-\tfrac{1}{2}-\tfrac{u}{2c}\Big)\big(\rho
_T(w)+u\rho_T'(w)\big)\,du
+ o(c^{2}).
\end{aligned}
\]
The zeroth-order terms cancel, while the first-order terms give  
\[
\int_{0}^{c}\!\Big(\tfrac{1}{2}-\tfrac{u}{2c}\Big)u\,du
+\int_{-c}^{0}\!\Big(-\tfrac{1}{2}-\tfrac{u}{2c}\Big)u\,du
= \frac{c^{2}}{6}.
\]
Thus, \(\mathbb{E}\!\left[g_w(Z(T))\right]-\mathbb{E}\!\left[g_w^c(Z(T))\right] = \tfrac{c^{2}}{6}\rho_T'(w) + o(c^{2})\).
\section{Proof of Proposition \ref{smf}}
\label{proof4}
Using the standard weak error representation (see, e.g.,~\cite{szepessy2001adaptive}),
\begin{equation}
\mathbb{E}\big[g_{w}^{c}(Z(T))\big]
 - \mathbb{E}\big[g_{w}^{c}(Z^{d}(T))\big]
= \int_{0}^{T}
   \mathbb{E}\!\left[(f(Z(s)) - f^{d}(Z(s)))\,
   \partial_{z}u^{c}(s,\boldsymbol{X}(s),Z(s))\right] ds.
\end{equation}
Fix \(s\in(0,T)\), and denote 
\begin{equation}
Y(s)
:= (f(Z(s))-f^{d}(Z(s)))\,\partial_{z}u^{c}(s,\boldsymbol{X}(s),Z(s)).
\end{equation}
Then, by the law of total expectation,
\begin{equation}
\mathbb{E}[Y(s)]
=
\mathbb{E}\!\left[
\mathbb{E}\!\left[
\bigl(f(Z(s)) - f^{d}(Z(s))\bigr)\,
\partial_{z} u^{c}\bigl(s,\boldsymbol{X}(s),Z(s)\bigr)
\,\Big|\, Z(s)
\right]
\right].
\end{equation}
Therefore,
\begin{equation}
\mathbb{E}[Y(s)]
= \int_{\mathbb{R}}
    (f(z)-f^{d}(z))\,h^c_{s}(z)\,\rho_{s}(z)\,dz.
\end{equation}
where the deterministic function \(h^c_{s}\colon \mathbb{R}\to\mathbb{R}\) is defined by
\begin{equation}
h^c_{s}(z)
:= \mathbb{E}\!\left[\partial_{z}u^{c}(s,\boldsymbol{X}(s),Z(s))\mid Z(s)=z\right].
\end{equation}
Let $H^c_{s}(z):=h^c_{s}(z)\rho_{s}(z)$. Since the density $\rho_s$ is supported in $(0,s)$, we may insert
an indicator function:
\begin{equation}
H^c_s(z) = H^c_s(z)\,\mathbf{1}_{\{0<z<s\}}.
\end{equation}
\(f - f^{d}\) is supported in \((\gamma-d,\gamma+d)\), and hence
\begin{equation}
\mathbb{E}[Y(s)]
= \int_{\gamma-d}^{\gamma+d}
    (f(z)-f^{d}(z))\,H^c_{s}(z)\,\mathbf{1}_{\{z<s\}}\,dz .
\end{equation}
Integrating over $s$ gives
\begin{equation}
\Delta_c^d
= \int_0^T \mathbb{E}[Y(s)]\,ds
= \int_0^T \int_{\gamma-d}^{\gamma+d}
    (f(z)-f^{d}(z))\,H^c_{s}(z)\,\mathbf{1}_{\{z<s\}}\,dz\,ds .
\end{equation}
Suppose that the following assumption holds:\\
\begin{assum}
\label{ass1}
\begin{equation}
\int_z^T |H^{c}_{s}(z)| < \infty
\qquad\text{for all } z\in(\gamma-d,\gamma+d).
\end{equation}
\end{assum}
By Fubini’s theorem, we may interchange the order of integration:
\begin{equation}
\Delta_c^d
= \int_{\gamma-d}^{\gamma+d} (f(z)-f^d(z))
   \left( \int_0^T H^c_s(z)\,\mathbf{1}_{\{z<s\}}\,ds \right) dz
= \int_{\gamma-d}^{\gamma+d} (f(z)-f^d(z))\,F(z)\,dz ,
\end{equation}
where
\begin{equation}
F(z) := \int_0^T H^c_s(z)\,\mathbf{1}_{\{z<s\}}\,ds
      = \int_z^T H^c_s(z)\,ds,
      \qquad z\in(\gamma-d,\gamma+d).
\end{equation}
Suppose that the following assumptions hold:
\begin{assum}
\label{ass2} For each \(s \in [0,T]\), the map \(z \mapsto H_s(z)\) is \(C^{1}\big([\gamma- d,\;\gamma + d]\big)\).
\end{assum}
\begin{assum}
\label{ass3} There exists an integrable function \(\hat{H}:[0,T] \to \mathbb{R}_+\) such that
\begin{equation}
\left| \frac{\partial}{\partial z} H_s(z) \right| \le \hat{H}(s)
\quad \text{for all } z \in [\gamma- d,\; \gamma + d].
\end{equation}
\end{assum}
Using Leibniz, $F$ is $C^1$ in a
neighborhood of $\gamma$ and 
\begin{equation}
F'(\gamma)
= -\,H^c_{\gamma}(\gamma)
  + \int_{\gamma}^{T} \partial_z H^c_s(\gamma)\,ds .
\end{equation}
A Taylor expansion yields
\begin{equation}
F(z) = F(\gamma) + F'(\gamma)(z-\gamma) + o(z),
\end{equation}
Thus,
\begin{equation}
\Delta_c^d
= \int_{\gamma-d}^{\gamma+d} (f(z)-f^d(z))
   \big(F(\gamma) + F'(\gamma)(z-\gamma) + o(z)\big)\,dz .
\end{equation}
Let $u=z-\gamma$ and define
\begin{equation}
\psi(u) := f(\gamma+u)-f^d(\gamma+u), \qquad u\in[-d,d].
\end{equation}
Then
\begin{equation}
\Delta_c^d
= \int_{-d}^d \psi(u)
   \Big(F(\gamma)+F'(\gamma)u+o(\gamma+u)\Big)\,du.
\end{equation}
By construction of $f^d$, one has
\begin{equation}
\int_{-d}^d \psi(u)\,du = 0,
\qquad
\int_{-d}^d u\,\psi(u)\,du =-\frac{d^2}{6},
\end{equation}

\begin{equation}
\Delta_c^d
= F(\gamma)\!\int_{-d}^d \psi(u)\,du
  + F'(\gamma)\!\int_{-d}^d u\,\psi(u)\,du
  + o(d^2).
\end{equation}
The first term vanishes. Therefore:

\begin{equation}
\Delta_c^d = -F'(\gamma)\,\frac{d^2}{6}+ o(d^2)
\end{equation}
which proves the claim.
\bibliographystyle{unsrt}
\bibliography{References.bib}
\end{document}